%% file: journal.tex
\documentclass[10pt,reqno]{amsart}

\usepackage{graphicx}
\usepackage{caption}
\usepackage{subcaption}
\usepackage{amsmath}
\usepackage{amsaddr}
\usepackage{amssymb}
\usepackage{hyperref}
\usepackage[textheight=20cm, textwidth=14cm]{geometry}

\usepackage{chngcntr}
\usepackage{apptools}
\usepackage{verbatim}
\AtAppendix{\counterwithin{theorem}{section}}
\AtAppendix{\counterwithin{proposition}{section}}

\newtheorem{theorem}{Theorem}

\newtheorem{lemma}{Lemma}

\theoremstyle{definition}

\theoremstyle{remark}
\newtheorem{remark}{Remark}
\newtheorem{corollary}{Corollary}

\numberwithin{equation}{section}

\newcommand\restr[2]{{
		\left.\kern-\nulldelimiterspace 
		#1 
		\vphantom{\big|} 
		\right|_{#2} 
}}

\numberwithin{theorem}{section}
\numberwithin{lemma}{section}
\numberwithin{remark}{section}
\numberwithin{exercise}{section}
\numberwithin{corollary}{section}

\begin{document}

\title[The Poincar\'{e}-Bendixson theory in Banach spaces]
{The Poincar\'{e}-Bendixson theory for certain compact semiflows in Banach spaces}

\author{Mikhail Anikushin}
\address{Department of
	Applied Cybernetics, Faculty of Mathematics and Mechanics,
	Saint Petersburg State University, Universitetskiy prospekt 28, Peterhof, St. Petersburg 198504, Russia}
\address{Euler International Mathematical Institute, St. Petersburg Department of Steklov Mathematical Institute of Russian Academy of Sciences, 27 Fontanka, St. Petersburg 191011, Russia}
\email{demolishka@gmail.com}
\date{\today}
\thanks{The reported study was funded by RFBR according to the research project № 20-31-90008; by a grant in the subsidies form from the federal budget for the creation and development of international world-class math centers, agreement between MES RF and PDMI RAS No. 075-15-2019-1620; by V.~A.~Rokhlin grant for young mathematicians of St.~Petersburg.}




\begin{abstract}
We study semiflows satisfying a certain squeezing condition with respect to a quadratic functional in some Banach space. Under certain compactness assumptions from our previous results it follows that there exists an invariant manifold, which is under more restrictive conditions is an inertial manifold. In the case of a two-dimensional manifold we obtain an analog of the Poincar\'{e}-Bendixson theorem on the trichotomy of $\omega$-limit sets. Moreover, we obtain conditions for the existence of an orbitally stable periodic orbit. Our approach unifies a series of papers by R.~A.~Smith, establishes their connection with the theory of inertial manifolds and opens a new perspective of applications. To verify the squeezing property in applications we use recently developed versions of the frequency theorem, which guarantee the existence of the required quadratic functional if some frequency-domain condition is satisfied. We present applications of our results for nonlinear delay equations in $\mathbb{R}^{n}$ and semilinear parabolic equations and discuss perspectives of applications to parabolic problems with delay and boundary controls.
\end{abstract}

\maketitle


\input{intro}
\input{preliminaries}
\input{Proof1}
\input{StabilityPart}
\input{ApplicationsDelay}
\input{ApplicationsParab}
\input{ApplicationsOther}
\input{Acknowledgements}


\bibliographystyle{amsplain}

\end{document}

%% file: intro.tex
\section{Introduction}
We start with a precise statement of our main results.

Let $\mathbb{E}$ be a real Banach space and $\varphi^{t} \colon \mathbb{E} \to \mathbb{E}$, where $t \geq 0$, be a semiflow on $\mathbb{E}$, i.~e.
\begin{description}
	\item[1)] $\varphi^{0}(v)=v$ for all $v \in \mathbb{E}$
	\item[2)] $\varphi^{t+s}(v)=\varphi^{t}(\varphi^{s}(v))$ for all $v \in \mathbb{E}$ and $t, s \geq 0$.
	\item[3)] The map $\mathbb{R}_{+} \times \mathbb{E} \to \mathbb{E}$ defined as $(t,v) \mapsto \varphi^{t}(v)$ is continuous.
\end{description}
For the sake of brevity we sometimes denote the semiflow by $\varphi$. 

Let $\langle v , f \rangle:= f(v)$ denote the pairing between $v \in \mathbb{E}$ and $f \in \mathbb{E}^{*}$. A bounded linear operator $P \colon \mathbb{E} \to \mathbb{E}^{*}$, i.~e. $P \in \mathcal{L}(\mathbb{E};\mathbb{E}^{*})$, is called \textit{symmetric} if $\langle v , P w \rangle = \langle w, P v \rangle$ for all $v, w \in \mathbb{E}$. For a given subspace $\mathbb{L} \subset \mathbb{E}$ we say that $P$ is \textit{positive} (resp. \textit{negative}) on $\mathbb{L}$ if $\langle v , P v \rangle > 0$ (resp. $\langle v , P v \rangle < 0$) for all non-zero $v \in \mathbb{L}$.

Let $\mathbb{E}$ be continuously embedded into some Hilbert space\footnote{In fact, we need only the Banach space structure on $\mathbb{H}$, but it will be convenient for applications and discussions to assume that $\mathbb{H}$ is a Hilbert space.} $\mathbb{H}$. We identify elements of $\mathbb{E}$ and $\mathbb{H}$ under the embedding. We denote the norms in $\mathbb{E}$ and $\mathbb{H}$ by $\|\cdot\|_{\mathbb{E}}$ and $|\cdot|_{\mathbb{H}}$ respectively. Now our main conditions imposed on $\varphi$ can be introduced as follows.
\begin{description}
	\item[\textbf{(H1)}] There is a continuous linear operator $P \in \mathcal{L}(\mathbb{E};\mathbb{E}^{*})$, symmetric and such that $\mathbb{E}$ splits into the direct sum of some subspaces $\mathbb{E}^{+}$ and $\mathbb{E}^{-}$, i.~e. $\mathbb{E} = \mathbb{E}^{+} \oplus \mathbb{E}^{-}$, such that $P$ is positive on $\mathbb{E}^{+}$ and negative on $\mathbb{E}^{-}$.
	\item[\textbf{(H2)}] For some integer $j \geq 0$ we have $\dim \mathbb{E}^{-} = j$.
	\item[\textbf{(H3)}] For $V(v):=\langle v,Pv \rangle$ and some numbers $\delta>0$, $\tau_{V} \geq 0$ and $\nu>0$ we have
	\begin{equation}
	\begin{split}
	\label{EQ: SqueezingProperty}
	e^{2\nu r}V(\varphi^{r}(v_{1})-\varphi^{r}(v_{2}))-e^{2\nu l} V(\varphi^{l}(v_{1})-\varphi^{l}(v_{2})) \leq \\ \leq -\delta\int_{l}^{r} e^{2\nu s} |\varphi^{s}(v_{1})-\varphi^{s}(v_{2})|^{2}_{\mathbb{H}}ds,
	\end{split}	
	\end{equation}
	satisfied for every $v_{1},v_{2} \in \mathbb{E}$ and $0 \leq l \leq r$ such that $r -l \geq \tau_{V}$.
\end{description}
\begin{remark}
	Under hypotheses \textbf{(H1)} and \textbf{(H2)} the set $\mathcal{C}_{V}:=\{ \langle v,Pv \rangle \leq 0 \}$ is a $j$-\textit{dimensional quadratic cone} in $\mathbb{E}$. From \textbf{(H3)} we have that the semiflow is strictly monotone w.~r.~t. the cone $\mathcal{C}_{V}$. Namely, if $v_{1}-v_{2} \in \mathcal{C}_{V}$ for some distinct points $v_{1},v_{2} \in \mathbb{E}$, then $\varphi^{t}(v_{1})-\varphi^{t}(v_{2}) \in \operatorname{Int}\mathcal{C}_{V}$ for all $t > 0$. For $j=1$ the quadratic cone $\mathcal{C}_{V}$ is a union of two convex closed cones, say $\mathcal{C}^{+}_{V}$ and $\mathcal{C}^{-}:=-\mathcal{C}^{+}_{V}$. In this case \textbf{(H3)} gives the monotonicity of $\varphi$ w.~r.~t. the partial order given by $\mathcal{C}^{+}_{V}$. Such kind of monotonicity is considered in the classical theory of monotone dynamical systems (see, for example, D.~N.~Cheban and Z.~Lui \cite{Cheban2019}; H.~L.~Smith \cite{SmithHL2017}). For $j>1$ there is no such order (due to the lack of convexity). However, the cone $\mathcal{C}_{V}$ (a cone of rank $j$ in the terminology of \cite{Sanches2009,FengWangWu2017}) defines a pseudo-order and the mentioned monotonicity also leads to certain limitations for the semiflow. Semiflows, which are monotone w.~r.~t. such high-rank cones, were studied in finite-dimensional spaces  by L.~A.~Sanchez \cite{Sanches2009} and in the context of Banach spaces by L.~Feng et al. \cite{FengWangWu2017}. In particular, in these works weaker analogs of the Poincar\'{e}-Bendixson theorem were obtained. However, it seems that the stability results, which we present below, along with some other topological consequences cannot be deduced from the abstract pseudo-monotonicity used in \cite{Sanches2009,FengWangWu2017}. In fact, our approach is a generalization of various theories of inertial manifolds and besides the monotonicity there is also a squeezing property contained in \textbf{(H3)} (see the discussion below).
\end{remark}
\begin{remark}
	In applications to differential equations, \textbf{(H3)} can be verified by checking a balance inequality (called \textit{frequency-domain condition} or \textit{frequency inequality}) between the linear and nonlinear parts. This inequality usually contains some norm of a modified resolvent and the Lipschitz constant of the nonlinearity. In the case of semilinear parabolic equations, the simplest form of this inequality is known as the Spectral Gap Condition (see Section \ref{SEC: ReactionDiffusionSmith}). Sometimes we are interested only in the dynamics on a certain bounded invariant set $\mathcal{S}$ since globally nonlinearities may not be Lipschitz. In this case truncation procedures, which change nonlinearities outside $\mathcal{S}$ with preserving the Lipschitz constants, are used.
\end{remark}

Another main assumption is the following compactness property.
\begin{description}
	\item[\textbf{(COM)}] There is $\tau_{com}>0$ such that the map $\varphi^{\tau_{com}} \colon \mathbb{E} \to \mathbb{E}$ is compact.
\end{description}
We will also use some smoothing estimate described as follows.
\begin{description}
	\item[\textbf{(S)}] There is $\tau_{S} \geq 0$ and a constant $C_{S} > 0$ such that
	\begin{equation}
		\| \varphi^{\tau_{S}}(v_{1}) - \varphi^{\tau_{S}}(v_{2}) \|_{\mathbb{E}} \leq C_{S} |v_{1}-v_{2}|_{\mathbb{H}}
	\end{equation}
    for all $v_{1},v_{2} \in \mathbb{E}$.
\end{description}

For basic facts from the theory of dynamical systems we refer to the monograph of I.~Chueshov \cite{Chueshov2015}. For a point $v_{0} \in \mathbb{E}$ we denote its positive semi-orbit by $\gamma^{+}(v_{0})$, i.~e. $\gamma^{+}(v_{0})=\bigcup_{t \geq 0} \varphi^{t}(v_{0})$, and we denote its $\omega$-limit set by $\omega(v_{0}):=\bigcap_{s \geq 0}\overline{\bigcup_{t \geq s}\varphi^{t}(v_{0})}$. A \textit{complete trajectory} is a continuous map $v \colon \mathbb{R} \to \mathbb{E}$ such that the equality $v(t+s)=\varphi^{t}(v(s))$ holds for all $t \geq 0$ and $s \in \mathbb{R}$. In this case we say that $v(\cdot)$ is \textit{passing through $v(0)$}. If there is a unique complete trajectory $v(\cdot)$ passing through $v_{0}$ we also consider its $\alpha$-limit set $\alpha(v_{0}):=\bigcap_{s \leq 0}\overline{ \bigcup_{t \leq s}v(t) }$, its negative semi-orbit $\gamma^{-}(v_{0}) := \bigcup_{t \leq 0} v(t)$ and its complete orbit $\gamma(v_{0}) := \gamma^{+}(v_{0}) \cup \gamma^{-}(v_{0})$. We will sometimes call the $\omega$-limit set of a point $v_{0}$ the $\omega$-limit of its orbit $\gamma_{+}(v_{0})$ or its semi-trajectory $t \mapsto \varphi^{t}(v_{0})$.

One of our main results is the following theorem.
\begin{theorem}
	\label{TH: PBTrichotomy}
	Let the semiflow $\varphi$ satisfy \textbf{(H1)}, \textbf{(H2)} with $j=2$, \textbf{(H3)}, \textbf{(COM)} and \textbf{(S)}; then the $\omega$-limit set $\omega(w_{0})$ of any point $w_{0} \in \mathbb{E}$ with the bounded positive semi-orbit is one of the following:
	\begin{enumerate}
		\item[\textbf{(T1)}] A stationary point;
		\item[\textbf{(T2)}] A periodic orbit;
		\item[\textbf{(T3)}] A union of some set of stationary points $\mathcal{N}$ and a set of complete orbits whose $\alpha$- and $\omega$-limit sets lie in $\mathcal{N}$.
	\end{enumerate}
\end{theorem}
\begin{corollary}
	\label{COR: Th1CorAlphaLimit}
	Under the hypotheses of Theorem \ref{TH: PBTrichotomy} the trichotomy \textbf{(T1)}, \textbf{(T2)}, \textbf{(T3)} holds for the $\alpha$-limit set of any complete trajectory bounded in the past.
\end{corollary}
\begin{corollary}
	\label{COR: Th1CorIsolatedOrbit}
	Under the hypotheses of Theorem \ref{TH: PBTrichotomy} any isolated orbitally stable periodic orbit is asymptotically orbitally stable. 
\end{corollary}

Remind that a stationary point $v_{0} \in \mathbb{H}$ is called \textit{Lyapunov stable} if for every $\varepsilon>0$ there exists $\delta>0$ such that $\|\varphi^{t}(v)-v_{0}\|_{\mathbb{E}} < \varepsilon$ for all $t \geq 0$ provided that $\|v-v_{0}\|_{\mathbb{E}} < \delta$. 

Let $\mathcal{D} \subset \mathbb{E}$ be some set containing a given stationary point $v_{0}$. We say that $\mathcal{D}$ is a $k$-\textit{dimensional local unstable set} for $v_{0}$ if 
\begin{enumerate}
	\item[\textbf{(U1)}] $\mathcal{D}$ is a homeomorphic image of some open $k$-dimensional cube;
	\item[\textbf{(U2)}] For every point $w_{0} \in \mathcal{D}$ there is a unique complete trajectory $w(\cdot)$ passing through $w(0)=w_{0}$ and $w(t) \to v_{0}$ as $t \to -\infty$;
	\item[\textbf{(U3)}] For every $\varepsilon>0$ there is $\delta>0$ such that if $\|w_{0} - v_{0}\|_{\mathbb{E}}<\delta$ and $w_{0} \in \mathcal{D}$ then $\|\varphi^{t}(w)-v_{0}\|_{\mathbb{E}}<\varepsilon$ for all $t \leq 0$.
\end{enumerate}

We call a stationary point $v_{0} \in \mathbb{E}$ \textit{terminal} if either it is Lyapunov stable or there is a $2$-dimensional unstable set for $v_{0}$. In the latter case we call $v_{0}$ an \textit{unstable terminal point}. The role of terminal points is in the following.
\begin{corollary}
	\label{COR: TerminalPoints}
	Under the hypotheses of Theorem \ref{TH: PBTrichotomy} suppose that $\omega(w_{0})$ contains a terminal point $v_{0}$. Then $\omega(w_{0}) = \{ v_{0} \}$.
\end{corollary}

We call a bounded closed subset $\mathcal{A} \subset \mathbb{E}$ an \textit{attractor} if there exists an open set $\mathcal{U} \subset \mathbb{E}$ such that $\mathcal{A} \subset \mathcal{U}$ and for any $v_{0} \in \mathcal{U}$ the positive semi-orbit $\gamma^{+}(v_{0})$ is compact in $\mathbb{E}$ and $\omega(v_{0}) \subset \mathcal{A}$. We call any such set $\mathcal{U}$ a \textit{neighborhood} of the attractor $\mathcal{A}$.
\begin{theorem}
	\label{TH: PeriodicOrbitExis}
	Let the semiflow $\varphi$ satisfy \textbf{(H1)}, \textbf{(H2)} with $j=2$, \textbf{(H3)}, \textbf{(COM)} and \textbf{(S)}. Suppose there is an attractor $\mathcal{A}$ that either contains no stationary points or the only stationary points in $\mathcal{A}$ are unstable terminal points. Then the set $\mathcal{A}$ contains at least one orbitally stable periodic orbit.
\end{theorem}

Our approach to Theorems \ref{TH: PBTrichotomy} and \ref{TH: PeriodicOrbitExis} uses modifications and additions to some arguments from the papers of R.~A.~Smith \cite{Smith1994,Smith1994PB,Smith1992,Smith1987OrbStab} and the classical Poincar\'{e}-Bendixson theory. In particular, to stay in the context of abstract semiflows we use the topological results of O.~H\'{a}j{e}k on the existence of transversals for flows on two-dimensional topological manifolds \cite{Hajek1968}. Under the hypotheses of Theorem \ref{TH: PBTrichotomy} there exists a two-dimensional invariant topological manifold $\mathfrak{A}$, which attracts all compact trajectories (for $\mathbb{E}=\mathbb{H}$ this is shown in our previous work \cite{Anikushin2020Red}). In fact, if we strengthen the inequality in \textbf{(S)} to be uniform in $\tau_{S}$ from some segment, then the manifold $\mathfrak{A}$ possesses the exponential tracking property (and the exponent is determined by $\nu$ from \textbf{(H3)}), i.~e. every semi-trajectory (not necessarily bounded!) is attracted at an exponential rate by some trajectory from $\mathfrak{A}$ (the proof will be given in \cite{Anikushin2020Geom}). In this case Theorem \ref{TH: PBTrichotomy} is a direct corollary of the results of O.~H\'{a}j{e}k \cite{Hajek1968} and the exponential tracking. Here we give a proof, which is independent of this property. For particular classes of equations the existence of transversals can be obtained by the Lipschitz flow-box theorem applied to equations, which describe the dynamics of $\varphi$ on $\mathfrak{A}$ (the so-called \textit{inertial form}). But this will require additional constructions and hypotheses. It seems worth to stay in the abstract context to avoid unnecessary verifications in applications. Another reason to make some arguments independent on the inertial manifolds properties could be possible generalizations of the present theory for non-compact semiflows or semiflows satisfying weaker analogs of \textbf{(H3)}. Results in this direction are known for ODEs \cite{Burkin2014Hidden,LeoBurShep1996}, but require further developments for infinite-dimensional systems.

The presented theory gives a unified approach for the series of papers by R.~A.~Smith \cite{Smith1994,Smith1994PB,Smith1992,Smith1987OrbStab}, connects them with various theories of inertial manifolds and opens new perspectives of applications. This became possible due to recent developments of frequency-domain methods. An important step was done by A.~V.~Proskurnikov \cite{Proskurnikov2015} who relaxed the controllability assumption used in the frequency theorem of Likhtarnikov-Yakubovich \cite{Likhtarnikov1977}. Further, two new versions of the frequency theorem for delay \cite{Anikushin2020FreqDelay} and parabolic \cite{Anikushin2020FreqParab} equations obtained by the present author show that under the frequency inequality used by R.~A.~Smith in \cite{Smith1992,Smith1994,Smith1994PB} there exists an operator $P$ with properties \textbf{(H1)}, \textbf{(H2)}, \textbf{(H3)}. We should note that Smith abandonned his approach based on quadratic functionals \cite{Smith1987OrbStab} since he was unable to provide natural conditions for the existence of such functionals for infinite-dimensional problems \cite{Smith1992}. 

In \cite{Anikushin2020FreqParab} it is shown that the famous Spectral Gap Condition \cite{Temam1997} is the simplest form of the frequency inequality used by R.~A.~Smith. Moreover, in the recent survey of A. Kostianko et al. \cite{KostiankoZelikSA2020} it is shown that the Spatial Avering Principle suggested by J.~Mallet-Paret and G.~R.~Sell \cite{MalletParetSell1988IM} can also be considered within \textbf{(H3)}. Thus, these classical studies are included into a more general geometric theory of inertial manifolds based on quadratic Lyapunov functionals. We shall present this general theory in the forthcoming paper \cite{Anikushin2020Geom}. For applications of the present theory, i.~e. for construction of $2$-dimensional inertial manifolds, the most convenient conditions are given by the frequency theorem. Various frequency-domain conditions give a lot of flexibility in applications and allow to obtain nontrivial results as the papers of R.~A.~Smith show. However, this approach should not be considered as a panacea. Success can be achieved for systems in which the nonlinearity is sufficiently simple so that it can be adequately described by its Lipschitz constant over some invariant region or when the linear part ``dominates'' the nonlinearity. This seems to be the main intuition before one starts calculations in practice.

Nontrivial spaces $\mathbb{E}$ and $\mathbb{H}$ appear in the case of delay equations. Verification of the smoothing estimate in \textbf{(S)} and applications of the frequency theorem \cite{Anikushin2020FreqDelay} are linked with the construction of semigroups for such equations in an appropriate Hilbert space. In \cite{Anikushin2020Semigroups} the present author gave a simple approach for this problem. Previously known constructions of such semigroups, which are based on the theory of accretive operators (see, for example, the work of G.~F.~Webb \cite{Webb1976}), have some limitations in applications (a discussion of this is given in \cite{Anikushin2020Semigroups}).

A historical background on developments of the Poincar\'{e}-Bendixson theorem from the classical smooth version up to semi-flows on the plane is given by K.~Ciesielski \cite{Ciesielski2012}. A review of several works extending the Poincar\'{e}-Bendixson theory for certain high-dimensional ODEs is contained in the paper of B.~Li  \cite{Li1981}. Another review for ODEs is done by I.~M.~Burkin \cite{Burkin2015}, who especially treats the works of R.~A.~Smith and their connection with frequency-domain methods (see also the monograph of G.~A.~Leonov et al. \cite{LeoBurShep1996}). 

In \cite{Burkin2015} I.~M.~Burkin and N.~N.~Khien suggested an approach based on developments of R.~A.~Smith's theory, which allows to localize hidden attractors for ODEs (see also the pioneering paper of G.~A.~Leonov and N.~V.~Kuznetsov \cite{LeoKuz2013Hidden}, where another approach is suggested). We hope that our study will lead to discoveries of hidden attractors in infinite-dimensional problems.

It is worth mentioning the approach coming from the shape theory. For example, in the papers of S.~A.~Bogatyi and V.~I.~Gutsu \cite{BogatiyGutsu1989}, B.~G\"{u}nter and J.~Segal \cite{GunterSegal1993} it is described the shape of attracting compacta (and then one can apply some topological classification for sets on the plane, which have the described shape, to obtain a weaker analog of the Poincar\'{e}-Bendixson trichotomy). However, this approach highly relies on the attracting property (=asymptotic stability) that fails to hold in certain situations covered by the classical Poincar\'{e}-Bendixson theorem and not easy to check in practice.

For more specific classes of systems the conclusion of Theorem \ref{TH: PBTrichotomy} can be achieved with the use of some integer-valued Lyapunov function, which usually counts the number of zeros. In particular, in the work of J.~Mallet-Paret and G.~R.~Sell \cite{MalletParetSell1996} this approach is applied to certain systems with delay. This theory as well as the mentioned works on monotone semiflows \cite{Sanches2009,FengWangWu2017} do not provide any stability results as in Theorem \ref{TH: PeriodicOrbitExis}. Although these theories cover many interesting examples, where Theorem \ref{TH: PBTrichotomy} cannot be applied, sometimes it is possible to strengthen these results with the aid of the present theory. A nice example was given by R.~A.~Smith in \cite{Smith1992}, where he obtained conditions for the existence of an orbitally stable periodic orbit for the Goodwin system of delay equations, which is a monotone cyclic feedback system in the terminology of \cite{MalletParetSell1996}. We will return to this example in Section \ref{SEC: DelayEqs}.

Frequency-domain methods are more developed for studying autonomous and non~-~autonomous ODEs, where similar to \textbf{(H3)} assumptions can be used to study the existence of periodic \cite{LeoBurShep1996} and almost periodic solutions \cite{AnikushinRR2019} or dimension-like properties \cite{Anikushin2019Vestnik,AnikushinRR2019,KuzLeoReit2019}. Some infinite-dimensional developments of theses ideas are given by the present author \cite{Anikushin2020Geom,Anikushin2020Semigroups,Anikushin2020Red} and also by N.~Yu.~Kalinin and V.~Reitmann \cite{KalininReitmann2012}.

This paper is organized as follows. In Section \ref{SEC: Preliminaries} we expose several auxiliary facts concerned with the existence of inertials manifolds and the existence of transversals for flows on two-dimensional manifolds. In Section \ref{SEC: Trichotomy} we prove Theorem \ref{TH: PBTrichotomy} and the corresponding corollaries. Section \ref{SEC: OrbitalStability} is devoted to the proof of Theorem \ref{TH: PeriodicOrbitExis}. In Section \ref{SEC: DelayEqs} we consider applications of Theorems \ref{TH: PBTrichotomy} and \ref{TH: PeriodicOrbitExis} to delay equations in $\mathbb{R}^{n}$. In Section \ref{SEC: ReactionDiffusionSmith} we consider applications to semilinear parabolic equations. In Section \ref{SEC: OtherApps} we discuss applications to parabolic equations with delay and boundary controls and further developments of adjacent results.

%% file: preliminaries.tex
\section{Preliminaries}
\label{SEC: Preliminaries}
\subsection{Existence of the inertial manifold $\mathfrak{A}$}

Under \textbf{(H3)} a complete trajectory $v(\cdot)$ of the semi-flow $\varphi$ is called \textit{amenable} if
\begin{equation}
\int\limits_{-\infty}^{0} e^{2\nu s}|v(s)|^{2}_{\mathbb{H}} ds < +\infty.
\end{equation}

Let us define $\mathfrak{A}$ as the set of all $v_{0} \in \mathbb{H}$ such that there exists an amenable trajectory passing through $v_{0}$. We call the set $\mathfrak{A}$ the \textit{amenable set}. A modification of Lemma 1 in \cite{Anikushin2020Red} (proved for the case $\mathbb{E}=\mathbb{H}$) gives us the following property.
\begin{lemma}
	\label{LEM: AmenableLemma}
	Under \textbf{(H1)}, \textbf{(H3)} and \textbf{(S)} let $v_{1}(\cdot)$ and $v_{2}(\cdot)$ be two distinct amenable trajectories; then $V(v_{1}(t)-v_{2}(t)) < 0$ for all $t \in \mathbb{R}$.
\end{lemma}
In particular, from Lemma \ref{LEM: AmenableLemma} we have that the map $\varphi^{t} \colon \mathfrak{A} \to \mathfrak{A}$ is bijective.

Under \textbf{(H1)} and \textbf{(H2)} we define $\Pi \colon \mathbb{E} \to \mathbb{E}^{-}$ be defined as follows. For any $v \in \mathbb{E}$ the map $\mathbb{E}^{-} \ni w \mapsto \langle w,  Pv\rangle$ is a linear functional on the finite-dimensional space $\mathbb{E}^{-}$. Since $-V(\cdot)$ is positive-definite on $\mathbb{E}^{-}$, there exists a unique element $\Pi v \in \mathbb{E}^{-}$ such that $\langle w , Pv \rangle = \langle w, P \Pi v \rangle$ for all $w \in \mathbb{E}^{-}$. We say that $\Pi$ is the $V$-orthogonal projector onto $\mathbb{E}^{-}$. It can be easily verified that $\Pi$ is bounded, $\mathbb{E} = \operatorname{Ker}\Pi \oplus \mathbb{E}^{-}$, $P$ is positive on $\mathbb{E}^{+}$ and $V(v)=V(v^{+})+V(v^{-})$, where $v = v^{+} + v^{-}$ is the unique decomposition with $v^{+} \in \operatorname{Ker}\Pi$ and $v^{-} \in \mathbb{E}^{-}$. In other words, under \textbf{(H1)} and \textbf{(H2)} we can always assume that the subspaces $\mathbb{E}^{+}$ and $\mathbb{E}^{-}$ from \textbf{(H1)} are $V$-orthogonal in the given sense.

The above construction of $\Pi$ and smoothing estimate \textbf{(S)} leads to a generalization of Theorem 1 from \cite{Anikushin2020Red} as follows.
\begin{theorem}
	\label{TH: MainReductionTheorem}
	Let the semiflow $\varphi$ satisfy \textbf{(H1)}, \textbf{(H2)}, \textbf{(H3)}, \textbf{(COM)} and \textbf{(S)}; then either $\mathfrak{A} = \emptyset$ or $\Pi \colon \mathfrak{A} \to \mathbb{E}^{-}$ is a homeomorphism.
\end{theorem}
The proof follows the same arguments as in \cite{Anikushin2020Red}. A more general proof for compact cocycles in Banach spaces with the fibre-dependent operator $P$ is given in \cite{Anikushin2020Geom}.

If the hypotheses of Theorem \ref{TH: MainReductionTheorem} hold we define the map $\Phi \colon \mathbb{E}^{-} \to \mathfrak{A}$ by the relation $\Pi \Phi(\zeta) = \zeta$ for all $\zeta \in \mathbb{E}^{-}$. By Theorem \ref{TH: MainReductionTheorem}, the map $\Phi$ is a homeomorphism.

\begin{corollary}
	\label{COR: AmenableManifoldProp}
	Under the hypotheses of Theorem \ref{TH: MainReductionTheorem} suppose that $\mathfrak{A}$ is not empty; then we have
	\begin{enumerate}
		\item[\textbf{(A1)}] $\mathfrak{A}$ is an invariant $j$-dimensional topological manifold, i.~e. $\varphi^{t}(\mathfrak{A})=\mathfrak{A}$ for all $t \geq 0$.
		\item[\textbf{(A2)}] Any map $\varphi^{t}$, $t \geq 0$, is continuously invertible on $\mathfrak{A}$ and the restriction of $\varphi$ to $\mathfrak{A}$ can be extended to a flow on $\mathfrak{A}$.
		\item[\textbf{(A3)}] For any $v_{0} \in \mathbb{E}$ with the compact semi-orbit we have
		\begin{equation}
		\left\| \varphi^{t}(v_{0})-\Phi(\Pi \varphi^{t}(v_{0})) \right\|_{\mathbb{E}} \to 0 \text{ as } t \to +\infty.
		\end{equation}
	\end{enumerate}
\end{corollary}
\begin{proof}
	The property in \textbf{(A1)} follows from the definition of $\mathfrak{A}$. By Lemma \ref{LEM: AmenableLemma}, for any $t \geq 0$ the map $\varphi^{t} \colon \mathfrak{A} \to \mathfrak{A}$ is bijective and, by the Brouwer theorem on invariance of domain, it is also a homeomorphism. From this and the group property, we have that the map $\mathbb{R} \times \mathfrak{A} \ni (t,v) \mapsto \varphi^{t}(v)$, where for $t < 0$ the map $\varphi^{t} \colon \mathfrak{A} \to \mathfrak{A}$ is defined as the inverse to $\varphi^{-t} \colon \mathfrak{A} \to \mathfrak{A}$, is continuous. Thus, \textbf{(A2)} holds. The convergence in \textbf{(A3)} can be easily proved by contradiction.
\end{proof}

\subsection{Transversals of flows at non-stationary points}

Let $\xi^{t} \colon \mathcal{X} \to \mathcal{X}$, where $t \in \mathbb{R}$, be a flow on a complete metric space $\mathcal{X}$. Suppose $\varepsilon>0$ and a point $x_{0} \in \mathcal{X}$ are given. A subset $\mathcal{S} \subset \mathcal{X}$ is called an $\varepsilon$-\textit{section} of $\xi$ at $x_{0}$ if $x_{0} \in \mathcal{S}$, the set $\mathcal{U}=\mathcal{U}(\varepsilon,\mathcal{S}):=\bigcup_{t \in [-\varepsilon,\varepsilon]} \xi^{t}(S)$ is a topological neighbourhood\footnote{That is $x_{0}$ belongs to the interior of $\mathcal{U}$.} of $x_{0}$ in $\mathcal{X}$ and for every $y \in \mathcal{U}$ there exists a unique point $y_{0} \in \mathcal{S}$ and a unique time moment $t \in [-\varepsilon,\varepsilon]$ such that $\xi^{t}(y_{0})=y$. 

Now suppose $\mathcal{X}$ is a two-dimensional manifold. A set $\mathcal{T} \subset \mathcal{X}$ is called an $\varepsilon$-\textit{transversal} of $\xi$ at $x_{0}$ if it is simultaneously an $\varepsilon$-section and a homeomorphic image of a closed segment, for which $x_{0}$ corresponds to some of its interior points. We will often call $\mathcal{T}$ simply a \textit{transversal} of $\xi$ if for some $\varepsilon>0$ it is an $\varepsilon$-transversal of $\xi$ at some point $x_{0}$.

The following theorem is due to O. H\'{a}jek (see Chapter VII, Corollary 2.6 in \cite{Hajek1968}).
\begin{theorem}
	\label{TH: HajekTheorem}
	Let $\xi$ be a flow on some two-dimensional manifold $\mathcal{X}$; then for any non-stationary point $x_{0}$ and all sufficiently small $\varepsilon>0$ there exists an $\varepsilon$-transversal $\mathcal{T}$ of $\xi$ at $x_{0}$.
\end{theorem}

Let $\mathcal{T}$ be a $\varepsilon$-transversal of $\xi$ at $x_{0}$. It is clear that any closed connected subset of $\mathcal{T}$ is also a $\varepsilon$-transversal of $\xi$. Without loss of generality, we may assume that $\mathcal{T}$ is given by a homeomorphism $h \colon [-1,1] \to \mathcal{T}$ with $h(0)=x_{0}$. In this case the map $[-\varepsilon,\varepsilon] \times [-1,1] \to \mathcal{U}=\mathcal{U}(\varepsilon,\mathcal{T})$ defined as $(t,s) \mapsto \varphi^{t}(h(s))$ gives a homeomorphism of the cube $[-\varepsilon,\varepsilon] \times [-1,1]$ onto its image. By the Brouwer theorem on invariance of domain, the interior $(-\varepsilon,\varepsilon) \times (-1,1)$ is mapped onto an open subset of $\mathcal{X}$.

%% file: Proof1.tex
\section{Trichotomy of $\omega$-limit sets}
\label{SEC: Trichotomy}

In this section we suppose that the hypotheses of Theorem \ref{TH: PBTrichotomy} are satisfied. Note that \textbf{(COM)} implies that any bounded semi-trajectory is compact (see, for example, see Lemma 4 in \cite{Anikushin2020Red}). From Corollary \ref{COR: AmenableManifoldProp} we get that there is an invariant 2-dimensional manifold $\mathfrak{A}$, which is homeomorphic to $\mathbb{E}^{-}$. By $\xi^{t} \colon \mathfrak{A} \to \mathfrak{A}$, where $t \in \mathbb{R}$, we denote the flow on $\mathfrak{A}$ given by the restriction of $\varphi$ to $\mathfrak{A}$. For convenience, for $v \in \mathcal{A}$ and $t < 0$ we usually write $\varphi^{t}(v)$ instead of $\xi^{t}(v)$. Let $w_{0} \in \mathbb{E}$ be a point with the compact positive semi-orbit $\gamma^{+}(w_{0})$. Recall that in this case the $\omega$-limit set of $w_{0}$, i.~e. $\omega(w_{0})$, is a non-empty compact connected invariant subset.

For a point $v \in \mathbb{E}$ and a number $r > 0$ by $\mathcal{O}_{r}(v)$ we denote the ball of radius $r$ centered at $v$. Analogously, for a subset $\mathcal{C} \subset \mathbb{E}$ by $\mathcal{O}_{r}(\mathcal{C})$ we denote the $r$-neighbourhood of $\mathcal{C}$.

The following lemma is an adaptation of the well-known Bendixson lemma. It rely on the attracting property \textbf{(A3)} of $\mathfrak{A}$. 
\begin{lemma}
	\label{LEM: BendixsonLemma}
	Let $\mathcal{T}$ be an $\varepsilon$-transversal of $\xi$ and $v_{0} \in \omega(w_{0})$. Then for the corresponding complete trajectory $v(t)$ with $v(0)=v_{0}$ the trajectory $\Pi v(t)$ crosses $\mathcal{T}$ in at most one point.
\end{lemma}
\begin{proof}
	Supposing the contrary, we obtain two moments of time $t_{1},t_{2}$ with $t_{1} < t_{2}$ such that $v(t_{1}) \in \mathcal{T}$, $v(t_{2}) \in \mathcal{T}$ and $v(t) \notin \mathcal{T}$ for all $t \in (t_{1},t_{2})$. Consider the curve $\Gamma$ given by the part of the trajectory $v(t)$ for $t \in [t_{1},t_{2}]$ and the part of the transveral $\mathcal{T}$ between $v(t_{1})$ and $v(t_{2})$, which we denote by $\mathcal{T}_{t_{1}}^{t_{2}}$. Clearly, $\Gamma$ is a simple closed curve and, by the Jordan curve theorem, $\Gamma$ divides $\mathfrak{A}$, which is a homeomorphic image of the plane, into two parts: the interior (the bounded part) and the exterior (the unbounded part). From this there are only two possible cases, in which $\varphi^{t}$ maps $\mathcal{T}_{t_{1}}^{t_{2}}$ into one of these two parts (see Fig. \ref{Fig: Bendixson2Cases}). Let us consider the first case (Fig. \ref{fig:sub1}). The second one (Fig. \ref{fig:sub2}) can be treated analogously.
	
	Consider two points $v_{in}=v(t_{in})$ and $v_{out}=v(t_{out})$ corresponding to two moments of time $t_{in} > t_{2}$ and $t_{out} < t_{1}$ such that $v_{in}$ belongs to the interior and $v_{out}$ belongs to the exterior of $\Gamma$. By definition, we have $v_{in},v_{out} \in \omega(w_{0})$. Let $r>0$ be given such that the open balls of radii $r$ centred at $v_{in}$ and $v_{out}$ respectively do not intersect with the $r$-neighbourhood of $\Gamma$.
	
	Let $\mathcal{K} \subset \mathbb{E}$ be any compact set containing the semi-trajectory $\varphi^{t}(w_{0})$, $t \geq 0$, the set $\Phi(\Pi\varphi^{t}(w_{0}))$, $t \geq 0$, and the set $\mathcal{T}^{t_{2}}_{t_{1}}$. In particular, $\omega(w_{0}) \subset \mathcal{K}$. Put $T:=t_{in}-t_{out}$. By the continuity of the semi-flow $\varphi^{t}$, $t \geq 0$, there is $\delta>0$ such that $\|\varphi^{t}(v_{1})-\varphi^{t}(v_{2})\|_{\mathbb{E}} < r$ for all $t \in [0,T]$ provided that $v_{1},v_{2} \in \mathcal{K}$ and $\|v_{1}-v_{2}\|_{\mathbb{E}} < \delta$.
	
	In virtue of \textbf{(A3)} there is $t_{\delta}>0$ such that $\|\varphi^{t}(w_{0})-\Phi(\Pi \varphi^{t}(w_{0}) )\|_{\mathbb{E}} < \delta/2$ holds for all $t \geq t_{\delta}$. Since $v_{in},v_{out} \in \omega(w_{0})$ there are two moments of time $t^{(r)}_{in}>t_{\delta}$ and $t^{(r)}_{out}>t^{(r)}_{in}$ such that $\|\varphi^{t^{(r)}_{in}}(w_{0})-v_{in}\|_{\mathbb{E}}<r$, $\|\varphi^{t^{(r)}_{out}}(w_{0})-v_{out}\|_{\mathbb{E}}<r$, $\|\Phi(\Pi\varphi^{t^{(r)}_{in}}(w_{0}))-v_{in}\|_{\mathbb{E}}<r$ and $\|\Phi(\Pi\varphi^{t^{(r)}_{out}}(w_{0}))-v_{out}\|_{\mathbb{E}}<r$. Since the trajectory of $w_{0}$ is continuous there must be a moment of time $t_{0} \in (t^{(r)}_{in}, t^{(r)}_{out})$ with $\Phi(\Pi \varphi^{t_{0}}(w_{0})) \in \Gamma$. Denote $\Gamma_{t_{1}}^{t_{2}}:=\{ \varphi^{t}(v_{0}) \ | \ t \in [t_{1},t_{2}] \}$ Suppose $\Phi(\Pi \varphi^{t_{0}}(w_{0})) \in \mathcal{O}_{\delta/2}(\Gamma_{t_{1}}^{t_{2}})$, i.~e. $\|\Phi(\Pi \varphi^{t_{0}}(w_{0}))-\varphi^{\widetilde{t}}(v_{0})\|_{\mathbb{E}}<\delta/2$ for some $\widetilde{t} \in [t_{1},t_{2}]$. From this it follows that for $v_{1}:=\varphi^{t_{0}}(w_{0})$ and $v_{2}:=\varphi^{\widetilde{t}}(v_{0})$ we have $\|v_{1}-v_{2}\|_{\mathbb{E}} < \delta$, $v_{1},v_{2} \in \mathcal{K}$ and, consequently, $\|\varphi^{t}(v_{1})-\varphi^{t}(v_{2})\|_{\mathbb{E}} < r$ for $t \in [0,T]$. Thus for all $t_{0} \geq t_{\delta}$ each time we have $\varphi^{t_{0}}(w_{0}) \in \mathcal{O}_{\delta/2}(\Gamma_{t_{1}}^{t_{2}})$ the point $\varphi^{t+t_{0}}(w_{0})$ remains in $\mathcal{O}_{r}(\Gamma_{t_{1}}^{t_{2}})$ for $t \in [0,T]$ and, moreover, $\varphi^{t'+t_{0}}(w_{0}) \in \mathcal{O}_{r}(v_{in})$ for some $t' \in [0,T]$. So, for $t \geq t_{\delta}$ the curve $\Phi(\Pi \varphi^{t}(w_{0}) )$ cannot reach $\mathcal{O}_{r}(v_{out})$ crossing $\Gamma \cap \mathcal{O}_{\delta/2}(\Gamma_{t_{1}}^{t_{2}})$.
	
	Consider the remaining part of the $\varepsilon$-transversal, i.~e. $\mathcal{S}_{t_{1}}^{t_{2}}:=\mathcal{T}_{t_{1}}^{t_{2}} \setminus \mathcal{O}_{\delta/2}(\Gamma_{t_{1}}^{t_{2}})$. There is $r>d>0$ such that $\mathcal{O}_{d}\left(\varphi^{\varepsilon}(\mathcal{S}_{t_{1}}^{t_{2}})\right)$ lies in the bounded part of the plane and do not intersect $\Gamma$. For some $\delta_{1}>0$ we have $\|\varphi^{t}(v_{1})-\varphi^{t}(v_{2})\|_{\mathbb{E}}<d$ for all $t \in [0,\varepsilon]$ provided that $v_{1},v_{2} \in \mathcal{K}$ and $\|v_{1}-v_{2}\|_{\mathbb{E}} < \delta_{1}$. Let $t_{\delta_{1}}>0$ be such that $\|\varphi^{t}(w_{0})-\Phi(\Pi \varphi^{t}(w_{0}) )\|_{\mathbb{E}} < \delta_{1}$ for all $t \geq t_{\delta_{1}}$. So, if $\Phi(\Pi\varphi^{t_{0}}(w_{0})) \in \mathcal{S}_{t_{1}}^{t_{2}}$ for some $t_{0} \geq t_{\delta_{1}}$ we immediately have $\varphi^{t_{0}+\varepsilon}(w_{0}) \in \mathcal{O}_{d}(\varphi^{\varepsilon}(\mathcal{S}_{t_{1}}^{t_{2}}))$.
	
	Thus, if for some $t_{0} \geq \max\{t_{\delta},t_{\delta_{1}}\}$ we have $\varphi^{t_{0}}(w_{0}) \in \mathcal{O}_{r}(v_{in})$ then $\varphi^{t}(w_{0}) \notin \mathcal{O}_{r}(v_{out})$ for all $t \geq t_{0}$ that leads to a contradiction with $v_{out} \in \omega(w_{0})$.
\end{proof}

\begin{figure}
	\centering
	\begin{subfigure}{.5\textwidth}
		\centering
		\includegraphics[width=0.8\linewidth]{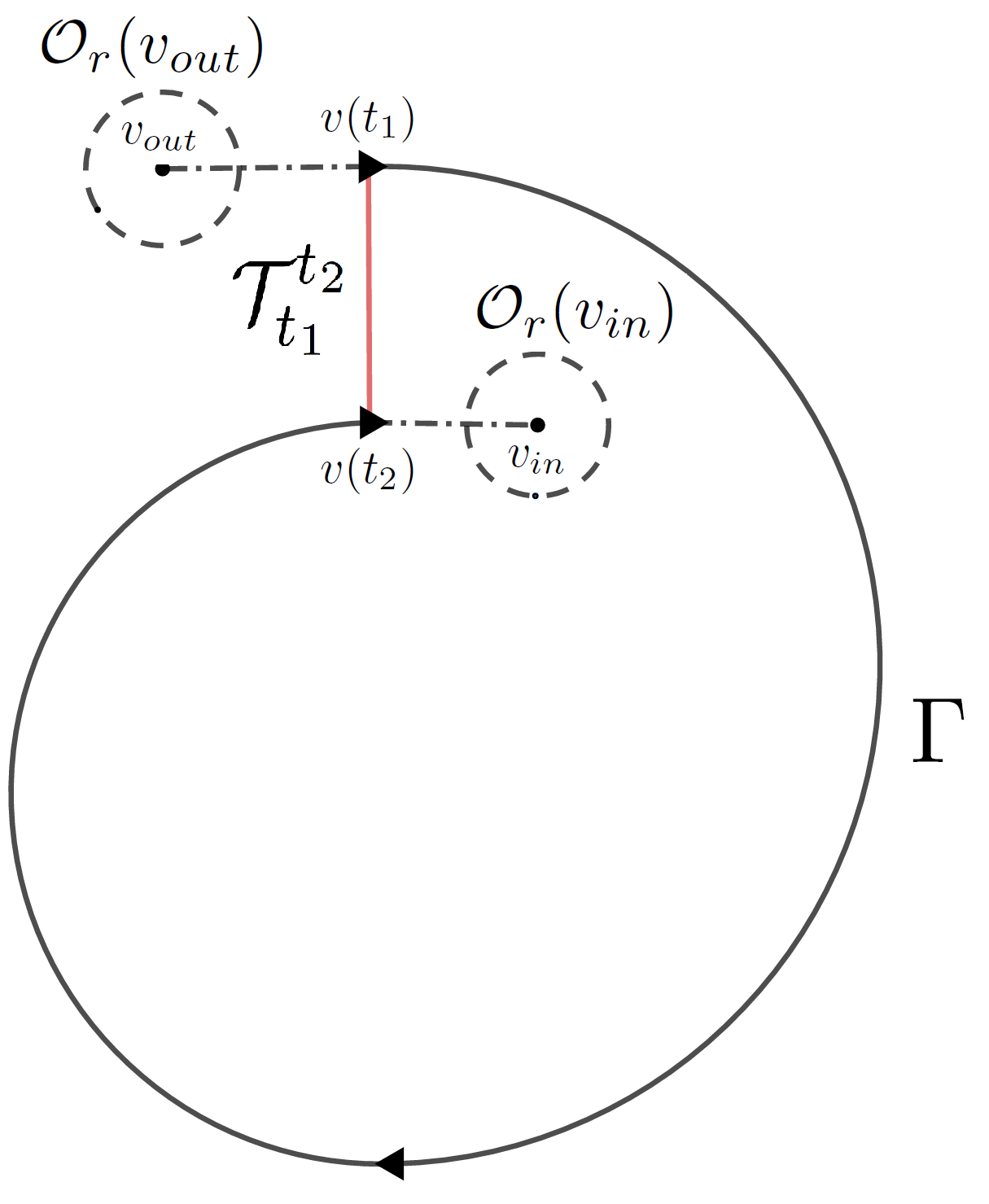}
		\caption{The transversal $\mathcal{T}_{t_{1}}^{t_{2}}$ (red) goes into \\ the interior of $\Gamma$.}
		\label{fig:sub1}
	\end{subfigure}%
	\begin{subfigure}{.5\textwidth}
		\centering
		\includegraphics[width=0.8\linewidth]{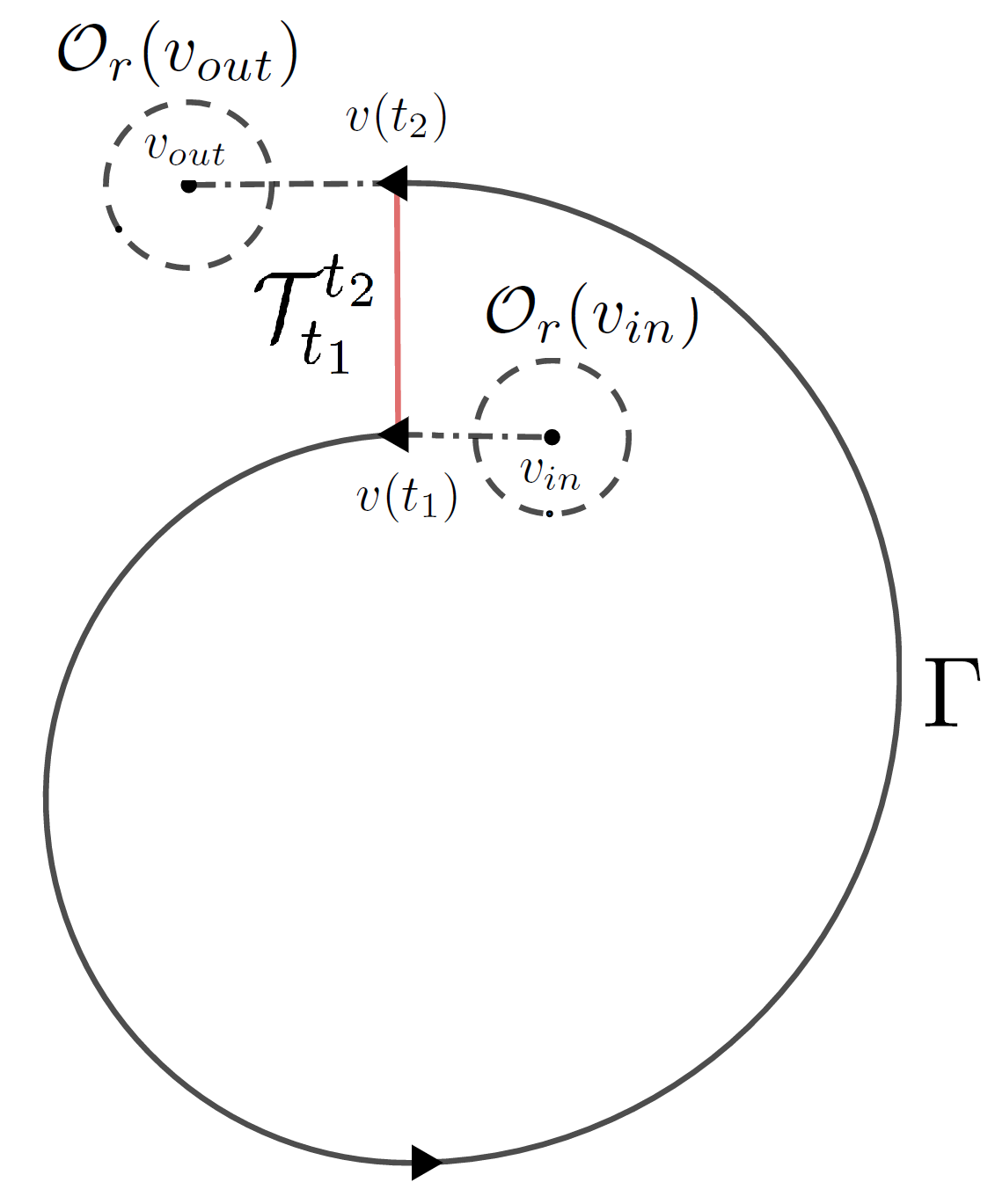}
		\caption{The transversal $\mathcal{T}_{t_{1}}^{t_{2}}$ (red) goes into \\ the exterior of $\Gamma$.}
		\label{fig:sub2}
	\end{subfigure}
	\caption{Two cases in Lemma \ref{LEM: BendixsonLemma}.}
	\label{Fig: Bendixson2Cases}
\end{figure}

The following lemma uses only the strict monotonicity given by \textbf{(H3)}.
\begin{lemma}
	\label{LEM: PeriodicSeparateLemma}
	Let $v_{1},v_{2} \in \mathbb{E}$ be two periodic points with distinct orbits $\gamma(v_{1})$ and $\gamma(v_{2})$ respectively; then there exists $\delta=\delta(v_{1},v_{2})>0$ such that for any $w_{0} \in \mathbb{E}$ with $\operatorname{dist}(w_{0},\gamma(v_{1})) < \delta$ we have $\Pi \varphi^{t}(w_{0}) \not\in \Pi\gamma(v_{2})$ for any $t \geq 0$. 
\end{lemma}
\begin{proof}
	Suppose the contrary, i.~e. that for every $\delta>0$ there exists a point $w_{\delta} \in \mathbb{E}$ such that $\operatorname{dist}(w_{\delta},\gamma(v_{1})) < \delta$ and there is $t_{\delta} > 0$ such that $\Pi \varphi^{t_{\delta}}(w_{\delta}) \in \Pi \gamma(v_{2})$. Let $v_{1,\delta} \in \gamma(v_{1})$ and $v_{2,\delta} \in \gamma(v_{2})$ be such that $\|w_{\delta}-v_{1,\delta}\|_{\mathbb{E}}<\delta$ and $\Pi \varphi^{t_{\delta}}(w_{\delta})=\Pi v_{2,\delta}$. Note that $t_{\delta} \to +\infty$ as $\delta \to 0+$. Let $\delta=\delta_{k}$, where $k=1,2,\ldots$, be a sequence tending to zero such that $v_{1,\delta_{k}} \to \overline{v}_{1} \in \gamma(v_{1})$, $v_{2,\delta_{k}} \to \overline{v}_{2} \in \gamma(v_{2})$ and $\varphi^{-t_{\delta}}(v_{2,\delta_{k}}) \to \widetilde{v}_{2} \in \gamma(v_{2})$ as $k \to \infty$. From \textbf{(H3)} with $r=t_{\delta}$, $l=0$, $v_{1}=w_{\delta}$ and $v_{2}=\varphi^{-t_{\delta}}(v_{2,\delta})$ we get (if $t_{\delta} \geq \tau_{V}$)
	\begin{equation}
	\label{EQ: LemmaPeriodicDeltaEq1}
	-e^{2\nu t_{\delta}}V( \varphi^{t_{\delta}}(w_{\delta}) - \varphi^{t_{\delta}}(\varphi^{-t_{\delta}}(v_{2,\delta})) ) + V(w_{\delta}-\varphi^{-t_{\delta}}(v_{2,\delta})) \geq 0.
	\end{equation}
	Since $\Pi \varphi^{t_{\delta}}(w_{\delta}) = \Pi v_{2,\delta}$, the first term in the left-hand side of \eqref{EQ: LemmaPeriodicDeltaEq1} is non-positive. Thus, we get
	\begin{equation}
	V(w_{\delta}-\varphi^{-t_{\delta}}(v_{2,\delta})) \geq 0.
	\end{equation}
	Our purpose is to show that the term in the left-hand side is negative for $\delta=\delta_{k}$ with sufficiently large $k$ that will lead to a contradiction. But this follows from the choice of $\delta_{k}$ and Lemma \ref{LEM: AmenableLemma} since as $k \to \infty$ we have
	\begin{equation}
	V(w_{\delta_{k}}-\varphi^{-t_{\delta_{k}}}(v_{2,\delta_{k}})) \to V(\overline{v}_{1},\widetilde{v}_{2}) < 0.
	\end{equation}
	Thus the lemma is proved.
\end{proof}

\begin{lemma}
	\label{LEM: PeriodicPointInLimitSet}
	If $v_{0} \in \omega(w_{0})$ is a periodic point then $\omega(w_{0}) = \gamma(v_{0})$.
\end{lemma}
\begin{proof}
To prove the statement we suppose the contrary, i.~e. that the set $\omega(w_{0}) \setminus \gamma(v_{0})$ is non-empty. Since $\omega(w_{0})$ is connected there exists a point $\overline{v} \in \gamma(v_{0})$ that is non-isolated from $\omega(w_{0}) \setminus \gamma(v_{0})$. Let $\mathcal{T} \subset \mathfrak{A}$ be an $\varepsilon$-transversal of $\restr{\varphi}{\mathfrak{A}}$ at $\overline{v}$ and put $\mathcal{U} = \bigcup_{t \in [-\varepsilon,\varepsilon]} \varphi^{t}(\mathcal{T})$. Let us show that if $\widetilde{v} \in \mathfrak{A}$ is sufficiently close to $\overline{v}$ then $\widetilde{v}$ is a periodic point. Indeed, any $\widetilde{v} \in \mathcal{U}$ has its trajectory crossing $\mathcal{T}$ at least once. If $\widetilde{v}$ is close enough to $\overline{v}$ then it must return to $\mathcal{U}$ after the period of $\overline{v}$ and by Lemma \ref{LEM: BendixsonLemma} it must cross $\mathcal{T}$ at the same point and, consequently, $\widetilde{v}$ is a periodic point. Since $\overline{v}$ is non-isolated from $\omega(w_{0}) \setminus \gamma(v_{0})$ it is a limit of a sequence of distinct periodic points from $\omega(w_{0})$. From this periodic points $\widetilde{v}, v_{sep} \in \omega(w_{0})$ can be chosen such that $\gamma(v_{sep})$ separates $\gamma(v_{0})$ and $\gamma(\widetilde{v})$. By Lemma \ref{LEM: PeriodicSeparateLemma} there is $\delta=\delta(v_{0},v_{sep})$ such that if $v \in \mathbb{H}$ and $\operatorname{dist}(v,\gamma(v_{0})) < \delta$ then $\Pi \varphi^{t}(v) \notin \Pi\gamma(v_{sep})$ for all $t \geq 0$. Since $v_{0}, \widetilde{v} \in \omega(w_{0})$ there must be moments of time $t_{0}>0$ and $\widetilde{t} > t_{0}$ such that $\operatorname{dist}(\varphi^{t_{0}}(w_{0}),\gamma(v_{0})) < \delta$ and $\varphi^{\widetilde{t}}(w_{0}) \in \mathcal{O}_{r}(\widetilde{v})$ with $r>0$ sufficiently small. But this gives a moment of time $t' \in (t_{0},\widetilde{t})$ with $\Pi \varphi^{t'}(w_{0}) \in \Pi\gamma(v_{sep})$ that contradicts to the previously established property. The lemma is proved.
\end{proof}

\begin{proof}[Proof of Theorem \ref{TH: PBTrichotomy}]
	Let $v_{0} \in \omega(w_{0})$ be a non-stationary point. Let us show that in this case either $\alpha(v_{0})$ and $\omega(v_{0})$ consist of stationary points or $\omega(w_{0})$ is a periodic orbit. If there is a non-stationary point $\widetilde{v}$ in any of these sets then the trajectory of $v_{0}$ must intersect a transversal at $\widetilde{v}$ infinitely many times. By Lemma \ref{LEM: BendixsonLemma} all these intersections coincide and the point $\widetilde{v}$ must be periodic. From Lemma \ref{LEM: PeriodicPointInLimitSet} it follows that $\omega(w_{0})=\gamma(\widetilde{v})$.
	
	So, either there is at least one stationary point in $\omega(w_{0})$ or $\omega(w_{0})$ is a periodic orbit (\textbf{(T2)} is realized). In the first case either the stationary point is the only point in $\omega(w_{0})$ (\textbf{(T1)} is realized) or we have \textbf{(T3)}.
\end{proof}

\begin{proof}[Proof of Corollary \ref{COR: Th1CorAlphaLimit}]
	Due to \textbf{(A1)} and \textbf{(A2)} the statement can be deduced from the Poincar\'{e}-Bendixson theorem for flows on the plane as in \cite{Hajek1968}. However, the key lemmas after obvious modifications and much simpler proofs show that the arguments used in the proof of Theorem \ref{TH: PBTrichotomy} still can be applied.
\end{proof}

\begin{proof}[Proof of Corollary \ref{COR: Th1CorIsolatedOrbit}]
	Let $\gamma_{0}$ be an isolated Lyapunov stable periodic orbit. There is $\varepsilon>0$ and $\delta>0$ such that $\varphi^{t}(v) \in \mathcal{O}_{\varepsilon}(\gamma_{0})$ for all $t \geq 0$ provided that $v \in \mathcal{O}_{\delta}(\gamma_{0})$ and the closure of $\mathcal{O}_{\varepsilon}(\gamma_{0})$ does not contain other periodic orbits and stationary points. Therefore, any $v \in \mathcal{O}_{\delta}(\gamma_{0})$ has a bounded (and, consequently, compact)  positive semi-orbit and its $\omega$-limit set in virtue of Theorem \ref{TH: PBTrichotomy} must coincide with $\gamma_{0}$. This shows the asymptotic orbital stability of $\gamma_{0}$.
\end{proof}

Let $v_{0}$ be a stationary point having a $k$-dimensional local unstable set $\mathcal{D}$. Under our assumptions it is clear that we always have $\mathcal{D} \subset \mathfrak{A}$. Therefore, any local unstable set has dimension $k \leq 2$.

\begin{proof}[Proof of Corollary \ref{COR: TerminalPoints}]
  	If $v_{0}$ is Lyapunov stable the statement is obvious. Let us consider the case when there is a $2$-dimensional unstable set $\mathcal{D}$ for $v_{0}$. By \textbf{(U1)} the set $\mathcal{D}$ is an open neighborhood of $v_{0}$ in $\mathfrak{A}$. To prove the statement it is sufficient to show that $v_{0}$ is the only stationary point in $\omega(w_{0})$. Indeed, if $v_{0}$ is the only stationary point and there is a non-stationary point $\widetilde{v} \in \omega(w_{0})$ then by Theorem \ref{TH: PBTrichotomy} it must be homoclinic to $v_{0}$ that contradicts to \textbf{(U3)}.
  	
  	Now suppose there is another stationary point $u_{0} \in \omega(v_{0})$. Let $r>0$ be a number and consider the closed ball of radius $r$ in $\mathfrak{A}$ (endowed with the metric from $\mathbb{E}$) centered at $v_{0}$, which we denote by $\mathcal{C}$, and the open ball of radius $r/2$ centered at $v_{0}$ in $\mathfrak{A}$, which we denote by $\mathcal{B}$. We assume that $r$ is chosen such that $\mathcal{C}$ is contained in $\mathcal{D}$ and therefore do not intersect with $u_{0}$. From \textbf{(U2)}, \textbf{(U3)} and the compactness of $\mathcal{C}$ we can find a number $T>0$ such that if $\widetilde{v} \in \mathcal{C}$ then $\varphi^{t}(\widetilde{v}) \in \mathcal{B}$ for all $t \leq -T$. From this it follows that
  	\begin{equation}
  	\label{EQ: TerminalProperty}
  	\text{if } \widetilde{v} \in \partial \mathcal{B} \text{ then } \varphi^{\widetilde{t}}(\widetilde{v}) \in \partial\mathcal{C} \text{ for some } \widetilde{t} \in (0,T).
  	\end{equation}
  	Let $d>0$ be such that $\mathcal{O}_{d}(\partial\mathcal{C}) \cap \mathcal{B} = \emptyset$. Consider a compact set $\mathcal{K}$ containing $\varphi^{t}(w_{0})$, $t \geq 0$, and $\Phi(\Pi \varphi^{t}(w_{0}))$, $t \geq 0$. Then there exists $\delta>0$ such that $\|\varphi^{t}(v_{1})-\varphi^{t}(v_{2})\|_{\mathbb{E}} < d$ for all $t \in [0,T]$ provided that $\|v_{1}-v_{2}\|_{\mathbb{E}} < \delta$ and $v_{1},v_{2} \in \mathcal{K}$. Using \textbf{(A3)} consider $t_{\delta}>0$ such that $|\varphi^{t}(w_{0})-\Phi(\Pi \varphi^{t}(w_{0}))| < \delta$ for all $t \geq t_{\delta}$. From this and \eqref{EQ: TerminalProperty} it follows that any time we have $\Phi(\Pi\varphi^{t_{0}}(w_{0})) \in \partial\mathcal{B}$ for some $t_{0} \geq t_{\delta}$ there is $\widetilde{t} \in (0,T)$ such that $|\varphi^{t_{0}+\widetilde{t}}(w_{0})-v_{0}| \geq r/2$. Since $u_{0} \in \omega(w_{0})$ there must be a time $t' \geq t_{\delta}$ when $\Phi(\Pi \varphi^{t'}(w_{0})) \notin \mathcal{C}$. Thus, for $t \geq t'$ the trajectory $\varphi^{t}(w_{0})$ cannot remain close to $v_{0}$ for the time intervals larger than $T$. This contradicts the fact that $v_{0} \in \omega(w_{0})$ and $v_{0}$ is stationary. So, $v_{0}$ is the only stationary point in $\omega(w_{0})$ and the lemma is proved.
\end{proof}

%% file: StabilityPart.tex
\section{Orbital stability}
\label{SEC: OrbitalStability}
In this section we also suppose that the hypotheses of Theorem \ref{TH: PBTrichotomy} are satisfied.

Recall that a periodic orbit $\gamma_{0}$ is called \textit{orbitally stable} if for every $\varepsilon>0$ there exists $\delta>0$ such that $\varphi^{t}(v) \in \mathcal{O}_{\varepsilon}(\gamma_{0})$ for all $t \geq 0$ provided that $v \in \mathcal{O}_{\delta}(\gamma_{0})$. In our context to study orbital stability of periodic orbits it is convenient to introduce the following definition.

A periodic orbit $\gamma_{0}$ is called \textit{amenable stable} if it is orbitally stable as a periodic orbit of the flow $\varphi$ restricted to $\mathfrak{A}$, i.~e. for every $\varepsilon>0$ there exists $\delta>0$ such that $\varphi^{t}(v) \in \mathcal{O}_{\varepsilon}(\gamma_{0})$ for all $t \geq 0$ provided that $v \in \mathcal{O}_{\delta}(\gamma_{0}) \cap \mathfrak{A}$.

The following lemma is a generalization of Theorem 3 from \cite{Smith1992}.

\begin{lemma}
	\label{LEM: AmenStOrbSt}
	Suppose a periodic orbit $\gamma_{0}$ is amenable stable; then it is orbitally stable.
\end{lemma}
\begin{proof}
	We will obtain a contradiction by assuming that the amenable stable periodic orbit $\gamma_{0}$ is not orbitally stable. Let $\delta_{k}>0$, $k=1,2,\ldots$, be a sequence tending to zero. Then for all sufficiently small $\varepsilon>0$ there exists a point $w^{(k)}_{\varepsilon} \in \mathcal{O}_{\delta_{k}}(\gamma_{0})$ and a moment of time $t_{k}>0$ such that $w^{(k)}_{\varepsilon}(t):=\varphi^{t}(w^{(k)}_{\varepsilon}) \in \mathcal{O}_{\varepsilon}(\gamma_{0})$ for all $t \in [0,t_{k})$ and $\operatorname{dist}(w^{(k)}_{\varepsilon}(t_{k}),\gamma_{0})=\varepsilon$. Since $\delta_{k} \to 0$, we must have $t_{k} \to +\infty$ as $k \to \infty$. Put $v^{(k)}_{\varepsilon}(t):=w^{(k)}_{\varepsilon}(t+t_{k})$ for $t \geq -t_{k}$. Using the boundedness of $v^{(k)}_{\varepsilon}(t)$ for $t \in (-\infty,0]$ and \textbf{(COM)} we can obtain a subsequence (we keep the same index $k$), which converges to some amenable trajectory $v_{\varepsilon}(\cdot)$ as $k \to +\infty$ (see Lemma 4 in \cite{Anikushin2020Red}). For $v_{\varepsilon}(\cdot)$ we have the properties
	\begin{enumerate}
		\item[(\textbf{$\ast$})] $\operatorname{dist}(v_{\varepsilon}(t),\gamma_{0}) \leq \varepsilon$ for all $t \in (-\infty,0)$;
		\item[(\textbf{$\ast\ast$})] $\operatorname{dist}(v_{\varepsilon}(0),\gamma_{0}) = \varepsilon$.
	\end{enumerate}
    If $\varepsilon$ is sufficiently small then the closure of $\mathcal{O}_{\varepsilon}(\gamma_{0})$ does not contain stationary points. From this, (\textbf{$\ast$}) and Corollary \ref{COR: Th1CorAlphaLimit} it follows that $\alpha(v_{\varepsilon}(0))$ must be a periodic trajectory $\gamma_{\varepsilon}$. Moreover, since $\gamma_{0}$ is amenable stable and (\textbf{$\ast\ast$}) holds, we must have $\gamma_{\varepsilon} \not= \gamma_{0}$. Thus, $\gamma_{0}$ is a non-isolated periodic orbit. From this it follows that for some $\varepsilon_{2}>\varepsilon_{1}>0$ and corresponding orbits $\gamma_{1}:=\gamma_{\varepsilon_{1}}$ and $\gamma_{2}:=\gamma_{\varepsilon_{2}}$ we have the property that $\gamma_{1}$ separates $\gamma_{0}$ and $\gamma_{2}$ on $\mathfrak{A}$. Let $\delta>0$ be given by Lemma \ref{LEM: PeriodicSeparateLemma} applied to $\gamma_{0}$ and $\gamma_{1}$, i.~e. if $v \in \mathcal{O}_{\delta}(\gamma_{0})$ then $\Pi \varphi^{t}(v) \not\in \Pi \gamma_{1}$ for all $t \geq 0$. From (\textbf{$\ast\ast$}) it is clear that $\gamma_{1}$ separates $v_{\varepsilon_{2}}(0)$ and $\gamma_{0}$ and, consequently, $\gamma_{1}$ separates $v^{(k)}_{\varepsilon_{2}}(0)$ and $\gamma_{0}$ for all sufficiently large $k$. Since $v^{(k)}_{\varepsilon_{2}}(-t_{k}) \in \mathcal{O}_{\delta_{k}}(\gamma_{0})$, $\gamma_{1}$ separates $v^{(k)}_{\varepsilon_{2}}(0)$ and $v^{(k)}_{\varepsilon_{2}}(-t_{k})$ for all sufficiently large $k$ and as a consequence there must by a time $t^{(k)}_{0} \in (-t_{k}, 0)$ such that $\Pi v^{(k)}_{\varepsilon_{2}}(t^{(k)}_{0}) \in \Pi \gamma_{1}$. This leads to a contradiction if $k$ is also chosen such that $\delta_{k} < \delta$.
\end{proof}

For a periodic orbit $\gamma$ in $\mathfrak{A}$ by $\mathcal{G}_{\gamma}$ we denote its interior (i.~e. the bounded component of $\mathfrak{A} \setminus \gamma$), which is well-defined by the Jordan curve theorem. Let $\gamma_{1}$ and $\gamma_{2}$ be two periodic orbits in $\mathfrak{A}$ we write $\gamma_{1} \leqslant \gamma_{2}$ iff $\mathcal{G}_{\gamma_{1}} \subset \mathcal{G}_{\gamma_{2}}$. Clearly, the relation $\leqslant$ defines a partial order on the set of periodic orbits. 

To describe the amenable stability the following concepts is useful. A periodic orbit $\gamma$ is called \textit{externally} (resp. \textit{internally}) \textit{stable} if either $\gamma$ is a limit of periodic orbits $\gamma_{k} \not=\gamma$, $k=1,2,\ldots$, with $\gamma \leq \gamma_{k}$ (resp. $\gamma_{k} \leq \gamma$) or there is a point $v_{0} \in \mathfrak{A} \setminus \operatorname{Cl}\mathcal{G}_{\gamma}$ (resp. $v_{0} \in \mathcal{G}_{\gamma}$) with $\omega(v_{0}) = \gamma$. The following lemma is obvious.

\begin{lemma}
	\label{LEM: ExtIntStable}
	A periodic orbit $\gamma$ which is both externally and internally stable is amenable stable.
\end{lemma}

Suppose $\mathcal{A}$ is an attractor satisfying the conditions of Theorem \ref{TH: PeriodicOrbitExis}. Note that for any of its neighborhood $\mathcal{U}$ we have that $\mathcal{U} \cap \mathfrak{A}$ is an open subset of $\mathfrak{A}$. 

Let $\operatorname{Per}(\mathcal{A})$ denote the set of all periodic points in $\mathcal{A}$. We have the following lemma.
\begin{lemma}
	\label{LEM: PeriodicSetClosed}
	The set $\operatorname{Per}(\mathcal{A})$ is non-empty and closed.
\end{lemma}
\begin{proof}
	Let $\mathcal{U}_{\mathcal{A}}$ be any neighbourhood of the attractor $\mathcal{A}$. Since unstable terminal points are separated from each other and cannot lie in $\omega$-limit sets of other amenable trajectories, there exist non-stationary points in $\mathcal{U}_{\mathcal{A}} \cap \mathfrak{A}$, which must lie in $\mathcal{A}$ and, consequently, due to Theorem \ref{TH: PBTrichotomy} must be attracted by some periodic trajectories. Thus, the set $\operatorname{Per}(\mathcal{A})$ is not empty.
	
	Now suppose $v_{0} \in \mathcal{A}$ is a limit of a sequence $v_{k} \in \operatorname{Per}(\mathcal{A})$, $k=1,2,\ldots$, of periodic points. Since $v_{0} \in \mathfrak{A}$, by the same arguments as above, from Theorem \ref{TH: PBTrichotomy} for some periodic trajectory $\gamma_{0}$ we must have that $\omega(v_{0})=\gamma_{0}$ and $\gamma_{0} \subset \operatorname{Per}(\mathcal{A})$. Let $\overline{v} \in \gamma_{0}$ be any point and for some sufficiently small $\varepsilon_{0}>0$ let $\mathcal{T} \subset \mathfrak{A}$ be an $\varepsilon_{0}$-transversal of $\varphi$ restricted to $\mathfrak{A}$ at $\overline{v}$. For every $t \in \mathbb{R}$ such that $\varphi^{t}(v_{0})$ belongs to the topological neighborhood $\mathcal{U}=\mathcal{U}(\varepsilon_{0},\mathcal{T})$ of $\overline{v}$ we must have an intersection with $\mathcal{T}$ for some $t_{0} \in (t-\varepsilon_{0},t+\varepsilon_{0})$. Therefore, there are infinitely many and arbitrary large times $t_{0}$ such that $\varphi^{t_{0}}(v_{0})$ belongs to $\mathcal{T}$. Let $\sigma>0$ be a period of $\gamma_{0}$ and consider $\varepsilon>0$ such that $\mathcal{O}_{\varepsilon}(\overline{v}) \cap \mathfrak{A} \subset \mathcal{U}$. Let $\delta>0$ be such that $\|\varphi^{t}(v)-\varphi^{t}(\overline{v})\|_{\mathbb{E}} < \varepsilon$ for $t \in [0,\sigma]$ provided that $v \in \mathcal{O}_{\delta}(\overline{v}) \cap \mathfrak{A}$. Suppose that for some times $0< t_{1} < t_{2}$ with $\varphi^{t_{1}}(v_{0}) \in \mathcal{O}_{\delta}(\overline{v})$ and $\varphi^{t_{2}}(v_{0}) \in \mathcal{O}_{\delta}(\overline{v})$ we have two distinct intersections with $\mathcal{T}$. Since $\varphi^{t_{2}}(v_{0}) \in \operatorname{Cl}\operatorname{Per}(\mathcal{A})$, a periodic orbit $\widetilde{\gamma}$ can be chosen such that $\gamma_{0} \leq \widetilde{\gamma}$ (or $\widetilde{\gamma} \leq \gamma_{0}$) and either $\widetilde{\gamma}$ separates $\varphi^{t_{1}}(v_{0})$ and $\varphi^{t_{2}}(v_{0})$ or $\widetilde{\gamma}$ separates $\varphi^{t_{2}}(v_{0})$ (or $\varphi^{t_{1}}(v_{0})$) and $\gamma_{0}$. In both cases we derive a contradiction. Therefore, $\varphi^{t}(v_{0})$ intersects $\mathcal{T}$ at the same point and consequently, $v_{0}$ is a periodic point.
\end{proof}

Let $\Gamma(\mathcal{A})$ denote the set of all periodic orbits in $\mathcal{A}$. Above we have defined a partial order $\leqslant$ on the set $\Gamma(\mathcal{A})$.
\begin{lemma}
	\label{LEM: PeriodicChain}
	Every chain $\mathcal{C}$ in $\Gamma$ has an upper bound $\gamma^{+}$ and a lower bound $\gamma^{-}$.
\end{lemma}
\begin{proof}
	To construct an upper bound we consider the set $\mathcal{G}^{+}:=\bigcup_{\gamma \in \mathcal{C}}\mathcal{G}_{\gamma}$. Since each of $\mathcal{G}_{\gamma}$'s is invariant and uniformly bounded, the set $\mathcal{G}^{+}$ is invariant and bounded. As a consequence, its boundary $\partial \mathcal{G}^{+}$ is a non-empty invariant compact set of $\mathfrak{A}$. It is easy to see that every point in $\partial \mathcal{G}^{+}$ is a limit of periodic points and, by Lemma \ref{LEM: PeriodicSetClosed}, $\partial \mathcal{G}^{+}$ consists of periodic points. Let $\gamma^{+} \subset \partial \mathcal{G}^{+}$ be a periodic orbit. By the previous argument, $\gamma^{+}$ is a limit of some sequence of $\gamma_{k} \in \mathcal{C}$, $k=1,2,\ldots$. From this it follows that $\mathcal{G}^{+}$ contains a point in the interior of $\gamma^{+}$ and since $\mathcal{G}^{+}$ is connected it must lie in this interior part. Therefore $\gamma \leqslant \gamma_{+}$ for every $\gamma \in \mathcal{C}$ that is required.
	
	Now we consider the set $\mathcal{G}_{-} = \bigcap_{\gamma \in \mathcal{C}} \operatorname{Cl}\mathcal{G}_{\gamma}$. Clearly, $\mathcal{G}_{-}$ is a compact invariant subset of $\mathfrak{A}$. As above, its boundary $\partial \mathcal{G}_{-}$ is a non-empty compact invariant set, which consists of periodic points. Let $\gamma^{-} \subset \partial \mathcal{G}_{-}$ be a periodic orbit. By a similar as above argument, $\gamma^{-} \leqslant \gamma$ for every $\gamma \in \mathcal{C}$ and, consequently, $\gamma^{-}$ is a lower bound.
\end{proof}

\begin{proof}[Proof of Theorem \ref{TH: PeriodicOrbitExis}]
By Lemma \ref{LEM: PeriodicChain} and Zorn's lemma the set $\Gamma(\mathcal{A})$ contains a maximal element $\gamma_{max}$, i.~e. there is no $\gamma \in \Gamma(\mathcal{A})$, $\gamma \not= \gamma_{max}$, such that $\gamma_{max} \leq \gamma$. Let us show that there exists at least one externally stable periodic orbit. If $\gamma_{max}$ is not externally stable then there is a point $v_{0} \in (\mathfrak{A} \setminus \operatorname{Cl}\mathcal{G}_{\gamma_{max}}) \cap \mathcal{U}_{\mathcal{A}}$, such that $\alpha(v_{0}) = \gamma_{max}$ and $\omega(v_{0}) \not= \gamma^{+}$. Since $\omega(v_{0}) \subset \mathcal{A}$ and the stationary points in $\mathcal{A}$ are unstable terminal, the set $\omega(v_{0})$ does not contain any stationary point and, consequently, by Theorem \ref{TH: PBTrichotomy} it is a periodic orbit $\widetilde{\gamma} \not= \gamma_{max}$. Since $\alpha(v_{0}) = \gamma_{max}$ the point $v_{0}$ lies in the exterior of $\widetilde{\gamma}$, i.~e. in $\mathfrak{A} \setminus \operatorname{Cl}\mathcal{G}_{\widetilde{\gamma}}$. Therefore, $\widetilde{\gamma}$ is externally stable.

Let $\Gamma^{ext}(\mathcal{A})$ denote the set of all externally stable periodic orbits in $\mathcal{A}$. By the above considerations, the set $\Gamma^{ext}(\mathcal{A})$ is not empty. Let $\mathcal{C}^{ext}$ be a chain in $\Gamma^{ext}(\mathcal{A})$. By Lemma \ref{LEM: PeriodicChain}, there exists a minimal element $\gamma_{min} \in \Gamma(\mathcal{A})$. Moreover, by the construction of $\gamma_{min}$ in the proof we have that either $\gamma_{min} \in \mathcal{C}^{ext}$ or $\gamma_{min}$ is a limit of a sequence of periodic orbits $\gamma_{k} \in \mathcal{C}^{ext}$ with $\gamma_{min} \leqslant \gamma_{k}$ and $\gamma_{k} \not= \gamma_{min}$, $k=1,2,\ldots$. In any of these cases we have $\gamma_{min} \in \Gamma^{ext}(\mathcal{A})$. By Zorn's lemma, there exists a minimal element $\gamma^{ext}_{min} \in \Gamma^{ext}(\mathcal{A})$, i.~e. there is no $\gamma^{ext} \in \Gamma^{ext}(\mathcal{A})$ such that $\gamma^{ext} \leqslant \gamma^{ext}_{min}$ and $\gamma^{ext} \not=\gamma^{ext}_{min}$.

Let us show that $\gamma^{ext}_{min}$ is internally stable. If $\gamma^{ext}_{min}$ is not a limit of a sequence of periodic orbits $\gamma_{k}$ with $\gamma_{k} \leq \gamma^{ext}_{min}$ and $\gamma_{k} \not= \gamma^{ext}_{min}$, $k=1,2,\ldots$, then there is $\delta>0$ such that $\mathcal{U}^{int}:=\mathcal{O}_{\delta}(\gamma^{ext}_{min}) \cap \mathcal{G}_{\gamma^{ext}_{min}}$ lies in $\mathcal{U}_{\mathcal{A}}$ and does not contain any periodic or stationary points. Then any point $v_{0} \in \mathcal{U}^{int}$ must satisfy $\omega(v_{0})=\gamma^{ext}_{min}$. Indeed, since all stationary points in $\mathcal{A}$ are unstable terminal, $\omega(v_{0})$ must be a periodic orbit $\widetilde{\gamma}$ with $\widetilde{\gamma} \subset \operatorname{Cl}\mathcal{G}_{\gamma^{ext}_{min}}$. If $\widetilde{\gamma} \not= \gamma^{ext}_{min}$ then $v_{0} \in \mathfrak{A} \setminus \operatorname{Cl}\mathcal{G}_{\widetilde{\gamma}}$ and, consequently, $\widetilde{\gamma}$ is an externally stable orbit with $\widetilde{\gamma} \leqslant \gamma^{ext}_{min}$ that contradicts to the minimality of $\gamma^{ext}_{min}$. So, the externally stable orbit $\gamma^{ext}_{min}$ is also internally stable and, by Lemmas \ref{LEM: ExtIntStable} and \ref{LEM: AmenStOrbSt}, it is orbitally stable. The proof is finished.
\end{proof}

%% file: ApplicationsDelay.tex
\section{Delay equations in $\mathbb{R}^{n}$}

Let us consider the following class of nonlinear delay differential equations in $\mathbb{R}^{n}$:
\begin{equation}
	\label{EQ: ClassicalDelayEquation}
	\dot{x}(t) = Ax_{t} + BF(Cx_{t}),
\end{equation}
where $x_{t}(\theta) := x(t+\theta)$, $\theta \in [-\tau,0]$, denotes the history segment; $\tau>0$ is a constant delay; $A \colon C([-\tau,0];\mathbb{R}^{n}) \to \mathbb{R}^{n}$, $B \colon \mathbb{R}^{m} \to \mathbb{R}^{n}$ and $C \colon C([-\tau,0];\mathbb{R}^{n}) \to \mathbb{R}^{r}$ are bounded linear operators and $F \colon \mathbb{R}^{r} \to \mathbb{R}^{m}$ is a nonlinear continuous function such that for some constant $\Lambda > 0$ we have
\begin{equation}
	\label{EQ: DelayLipschitz}
	|F(\sigma_{1}) - F(\sigma_{2})|_{1} \leq \Lambda |\sigma_{1}-\sigma_{2}|_{2} \text{ for all } \sigma_{1},\sigma_{2} \in \mathbb{R}^{r}, t \in \mathbb{R}.
\end{equation}
Here $|\cdot|_{1}, |\cdot|_{2}$ denote \underline{any} (this adds some flexibility in applications) inner products in $\mathbb{R}^{m}$ and $\mathbb{R}^{r}$ respectively.

From the classical theory \cite{Hale1977} it follows that for any $\phi_{0} \in C([-\tau,0];\mathbb{R}^{n})$ there is a unique classical solution $x(\cdot)=x(\cdot,\phi_{0}) \colon [-\tau,+\infty) \to \mathbb{R}^{n}$, i.~e. such that $x_{t_{0}} \equiv \phi_{0}$, $x(\cdot) \in C^{1}([0,+\infty);\mathbb{R}^{n})$ and $x(\cdot)$ satisfies \eqref{EQ: ClassicalDelayEquation} for $t \geq 0$. Let us define the semiflow $\varphi^{t} \colon C([-\tau,0];\mathbb{R}^{n}) \to C([-\tau,0];\mathbb{R}^{n})$, where $t \geq 0$, as $\varphi^{t}(\phi_{0}):=x_{t}(\cdot,s,\phi_{0})$, where $x_{t}(\theta,s,\phi_{0}) = x(t+\theta,s,\phi_{0})$ for $\theta \in [-\tau,0]$. For $\varphi$ we have that for any $T \geq 0$ there exists a constant $C_{L}=C_{L}(T)$ such that the inequality 
\begin{equation}
	\|\varphi^{t}(v_{1})-\varphi^{t}(v_{2})\|_{\mathbb{E}} \leq C_{L} \|v_{1}-v_{2}\|_{\mathbb{E}}.
\end{equation}
is satisfied for all $v_{1},v_{2} \in C([-\tau,0];\mathbb{R}^{n})$ and all $t \in [0,T]$. From this, \eqref{EQ: ClassicalDelayEquation} and the Arzel\'{a}–Ascoli theorem we immediately have that $\varphi$ satisfies \textbf{(COM)} with $\tau_{com} = \tau$.

Now let us consider the Hilbert space $\mathbb{H} := \mathbb{R}^{n} \times L_{2}(-\tau,0;\mathbb{R}^{n})$ and the continuous embedding of $\mathbb{E}:=C([-\tau,0];\mathbb{R}^{n})$ into $\mathbb{H}$ given by $\phi(\cdot) \mapsto (\phi(0),\phi(\cdot))$. As usual, we identify the elements of $\mathbb{E}$ and $\mathbb{H}$ under the embedding. Applying Theorem 1.1 from \cite{Anikushin2020Semigroups} we have the following theorem.
\begin{theorem}
	Under \eqref{EQ: DelayLipschitz} the semiflow $\varphi$ generated by equation \eqref{EQ: ClassicalDelayEquation} satisfies \textbf{(COM)} with $\tau_{com} = 2\tau$ and for any $T \geq 0$ there is a constant $C_{S}=C_{S}(T)>0$ such that
	\begin{equation}
		\| \varphi^{t}(v_{1}) - \varphi^{t}(v_{2})\|_{\mathbb{E}} \leq C_{S} |v_{1}-v_{2}|_{\mathbb{H}}
	\end{equation}
    is satisfied for all $v_{1},v_{2} \in \mathbb{E}$ and $t \in [\tau,\tau+T]$. In particular, \textbf{(S)} is satisfied with $\tau_{S} = \tau$.
\end{theorem}
In fact, in \cite{Anikushin2020Semigroups} it is proved that \eqref{EQ: ClassicalDelayEquation} is well-posed in $\mathbb{H}$ and the generalized solutions are uniquely defined by the semiflow $\varphi$ embedded from $\mathbb{E}$ into $\mathbb{H}$.

Now we are going to discuss verification of \textbf{(H1)}, \textbf{(H2)} and \textbf{(H3)}. By the Riesz–Markov–Kakutani representation theorem, there are functions $a(\cdot) \colon [-\tau,0] \to \mathbb{R}^{n}$ and $c(\cdot) \colon [-\tau,0] \to \mathbb{R}^{r}$ of bounded variation such that for any $\phi \in C([-\tau,0];\mathbb{R}^{n})$ we have
\begin{equation}
	A\phi = \int_{-\tau}^{0} d a(\theta)\phi(\theta) \text{ and } C\phi = \int_{-\tau}^{0} d c(\theta) \phi(\theta).
\end{equation}
It is known \cite{Hale1977} that the spectrum of the linear part of \eqref{EQ: ClassicalDelayEquation} is discrete and determined by the roots $p \in \mathbb{C}$ of
\begin{equation}
	\label{EQ: SpectrumDelayODE}
	\operatorname{det} \left( \int_{-\tau}^{0}d a(\theta)e^{p\theta} - p I \right) = 0.
\end{equation}
Moreover, for all $\nu \in \mathbb{R}$ equation \eqref{EQ: SpectrumDelayODE} has only finite number of roots with $\operatorname{Re}p > \nu$. If for some $\nu > 0$ there is no roots of \eqref{EQ: SpectrumDelayODE} with $\operatorname{Re}p = -\nu$ and exactly $j$ roots with $\operatorname{Re}p > -\nu$, then there are subspaces $\mathbb{E}^{s}(\nu)$ and $\mathbb{E}^{u}(\nu)$, where $\mathbb{E}^{u}(\nu)$ is the $j$-dimensional generalized eigenspace corresponding to the roots with $\operatorname{Re}p > - \nu$ and $\mathbb{E}^{s}(\nu)$ is a complementary subspace, such that $\mathbb{E} = \mathbb{E}^{s}(\nu) \oplus \mathbb{E}^{u}(\nu)$ (see \cite{Hale1977}).

Now let us consider the functions
\begin{equation}
	\alpha(p) := \int_{-\tau}^{0}e^{p\theta} d a(\theta)  \text{ and } \gamma(p) = \int_{-\tau}^{0} e^{p \theta} d c(\theta) .
\end{equation}
The matrix function $W(p) = \gamma(p) (\alpha(p) - p I)^{-1}B$, which may have poles in the roots $p \in \mathbb{C}$ of \eqref{EQ: SpectrumDelayODE}, is called the \textit{transfer function} of the triple $(A,B,C)$. In fact, it is a more convenient (for delay equations) form of the usual operator transfer function $W(p)=\hat{C}(\hat{A}-pI)^{-1}\hat{B}$, where $\hat{C}$, $\hat{A}$ and $\hat{B}$ are properly defined operators corresponding to $C,A,B$ from \eqref{EQ: ClassicalDelayEquation} (see \cite{Anikushin2020FreqDelay}).

The following theorem is contained in the results from Section 5 in \cite{Anikushin2020FreqDelay}.
\begin{theorem}
	\label{TH: DelayODEFreqCond}
	Suppose there is $\nu > 0$ such that equation \eqref{EQ: SpectrumDelayODE} has exactly $j$ roots with $\operatorname{Re}p > -\nu$ and does not have roots with $\operatorname{Re}p = -\nu$. Moreover, let the frequency-domain condition
	\begin{equation}
		\label{EQ: FreqDomainDelaySmith}
		|W(-\nu + i \omega)|_{\mathbb{C}^{m} \to \mathbb{C}^{r}} < \Lambda^{-1} \text{ for all } \omega \in \mathbb{R},
	\end{equation}
    where the matrix norm $|\cdot|_{\mathbb{C}^{m} \to \mathbb{C}^{r}}$ is induced by the inner products $|\cdot|_{1}$ and $|\cdot|_{2}$ chosen in \eqref{EQ: DelayLipschitz} for $\mathbb{R}^{m}$ and $\mathbb{R}^{r}$, be satisfied. Then there exists an operator $P \in \mathcal{L}(\mathbb{E};\mathbb{E}^{*})$ such that the semiflow satisfies \textbf{(H1)} with $\mathbb{E}^{+}= \mathbb{E}^{s}(\nu)$ and $\mathbb{E}^{-}=\mathbb{E}^{u}(\nu)$; \textbf{(H2)} with the given $j$ and \textbf{(H3)} with the given $\nu$, some $\delta > 0$ and $\tau_{V} =0$.
\end{theorem}

The frequency-domain condition \eqref{EQ: FreqDomainDelaySmith} for the case where $|\cdot|_{1}$ and $|\cdot|_{2}$ are Euclidean inner product was used by R.~A.~Smith \cite{Smith1992}. For delay and parabolic equations Smith did not use quadratic functionals (as in his earlier work \cite{Smith1987OrbStab}) since he was unable to provide adequate conditions for their existence. In \cite{Anikushin2020FreqDelay} it was firstly proved by the present author that under \eqref{EQ: FreqDomainDelaySmith} there exists an operator $P$ with the required properties. Thus, his results are contained in our theory. See also \cite{Anikushin2020FreqDelay} for a discussion about other frequency-domain conditions.

Thus, if the hypotheses of Theorem \ref{TH: DelayODEFreqCond} are satisfied with $j=2$, we can apply Theorem \ref{TH: PBTrichotomy} to get the Poincar\'{e}-Bendixson trichotomy of $\omega$-limit sets. For applications of Theorem \ref{TH: PeriodicOrbitExis}, which guarantees the existence of an orbitally stable periodic orbit, one usually constructs a bounded invariant region and check that all stationary points in this region are unstable terminal. The simplest case is when the system is dissipative. A simple dissipativity criterion for delay equations in $\mathbb{R}^{n}$ is given by R.~A.~Smith (see \cite{Smith1992}). To show that a stationary point is unstable terminal one should take the linearization of \eqref{EQ: ClassicalDelayEquation} at this point, check if there is only two unstable eigenvalues and then apply the Unstable Manifold Theorem for delay equations (see \cite{Hale1977}).

There are many ways to write a given delay equation in the form \eqref{EQ: ClassicalDelayEquation}, which is known in the control theory as the \textit{Lur'e form}. This leads to different frequency-domain conditions and gives the main flexibility in applications.

\subsection{A concrete example}

\label{SEC: DelayEqs}
Now let us consider applications of Theorem \ref{TH: PeriodicOrbitExis} to the system of Goodwin delay equations
\begin{equation}
\label{EQ: DelayGoodwin}
\begin{split}
\dot{x}_{1}(t)&=g(x_{n}(t-\tau) )-\lambda x_{1}(t),\\
\dot{x}_{2}(t)&=x_{1}(t)-\lambda x_{2}(t),\\
\ldots\\
\dot{x}_{n}(t)&=x_{n-1}(t)-\lambda x_{n}(t).
\end{split}
\end{equation}
Here $\tau,\lambda>0$ are positive constants and $g \colon \mathbb{R} \to \mathbb{R}$ is a continuous scalar-valued function, which is continuously differentiable in $(0,+\infty)$ and satisfies $0 > g'(\sigma) \geq -\kappa_{0}$ for some $\kappa_{0}>0$ and $\sigma \in (0,+\infty)$. Moreover, $g(\sigma) \to 0$ as $\sigma \to+\infty$.

System \eqref{EQ: DelayGoodwin} can be used as a model for certain biochemical reactions concerned with the protein synthesis. In this case the quantities $x_{j}$ represent concentrations of certain chemicals and therefore must satisfy $x_{j} \geq 0$. From our assumptions it is clear that the cone of non-negative functions in $C([-\tau,0];\mathbb{R}^{n})$ is invariant w.~r.~t. solutions of \eqref{EQ: DelayGoodwin}. In \cite{Smith1992} such systems (for $n=3$ and $n=4$) were studied by R.~A. Smith. For simplicity, we consider only the case $n=3$ to show some nuances arising in applications. In order to apply our results to study equation \eqref{EQ: DelayGoodwin} we redefine the nonlinearity $g$ outside some bounded invariant region.

For $\beta > 1$ we put $\sigma_{\beta}:= \left(\beta/\lambda\right)^{3} g(0)$ and consider the family of sets 
\begin{equation}
\mathcal{W}_{\beta}:=\{ \phi=(\phi_{1},\phi_{2},\phi_{3}) \in C([-\tau,0];\mathbb{R}^{3}) \ | \ (\beta\lambda)^{-j} g( \sigma_{\beta} ) < \phi_{j} < \left(\beta/\lambda\right)^{j} g(0) \}.
\end{equation}
From Lemma 14 in \cite{Smith1992} it follows that for every $\beta>1$ there is $\beta' \in (1,\beta)$ such that the set $\mathcal{A}:=\operatorname{Cl}\mathcal{W}_{\beta'}$, where the closure is taken in $\mathbb{E}$, can be considered as an attractor with the neighborhood $\mathcal{U}_{\mathcal{A}} := \mathcal{W}_{\beta}$. In \cite{Smith1992} it is also shown that $\mathcal{W}_{\beta}$ for $\beta > 1$ is an invariant set. For every $\beta \geq 1$ we consider the values (here $C\phi:=\phi_{3}(-\tau)$ for $\phi \in C([-\tau,0];\mathbb{R}^{3})$)
\begin{equation}
\label{EQ: DelayDelta}
\delta_{\beta}:=\sup_{\phi \in \operatorname{Cl}\mathcal{W}_{\beta}} g'(C \phi) = \sup_{\sigma \in C(\operatorname{Cl}\mathcal{W}_{\beta})} g'(\sigma) < 0.
\end{equation}
Thus, for $\sigma \in\operatorname{Cl}W_{\beta}$ we have $-\kappa_{0} \leq g'(\sigma) \leq \delta_{\beta}$ with $\delta_{\beta} < 0$. By $g_{\beta}$ we denote the function that coincides with $g$ on $\operatorname{Cl}\mathcal{W}_{\beta}$ and smoothly extended outside of $\operatorname{Cl}\mathcal{W}_{\beta}$ in such a way that the inequality $-\kappa_{0} \leq g'_{\beta}(\sigma) \leq \delta_{\beta}$ is satisfied for all $\sigma \in \mathbb{R}$. We will study the semi-flow $\varphi_{\beta}$ in $\mathbb{E}$ generated by
\begin{equation}
	\label{EQ: DelayGoodwinBeta}
	\begin{split}
		\dot{x}_{1}(t)&=g_{\beta}(x_{3}(t-\tau))-\lambda x_{1}(t),\\
		\dot{x}_{2}(t)&=x_{1}(t)-\lambda x_{2}(t),\\
		\dot{x}_{3}(t)&=x_{2}(t)-\lambda x_{3}(t),
	\end{split}
\end{equation}
Note that \eqref{EQ: DelayGoodwinBeta} is a monotone cyclic feedback system\footnote{To see this one should rename the variables $x_{1},x_{2},x_{3}$ to $x_{2},x_{1},x_{0}$ (in terms of \cite{MalletParetSell1996}) respectively.} in the terminology of \cite{MalletParetSell1996} and, consequently, the conclusion of Theorem \ref{TH: PBTrichotomy} holds for $\varphi_{\beta}$ for any parameters $\tau>0$ and $\lambda>0$. In fact, there is a global stability region (see \cite{Anikushin2020FreqDelay}) and thus for certain parameters there is no periodic orbits. It is interesting to reveal parameters for which the existence of periodic orbits is guaranteed.

Below we shall consider different Lur'e forms of \eqref{EQ: DelayGoodwinBeta}. Namely, for $\rho>0$ we consider
\begin{equation}
\label{EQ: GoodwinRhoBeta}
\dot{x}(t) = A_{\rho}x_{t} + B F_{\rho,\beta}(Cx_{t}),
\end{equation}
where $F(\sigma):=g_{\beta}(\sigma) + \rho \sigma$ and for $\phi \in C([-\tau,0];\mathbb{R}^{3})$, $y \in \mathbb{R}$ we have
\begin{equation}
	A_{\rho}\phi := \begin{bmatrix}
		-\lambda \phi_{1}(0) - \rho \phi_{3}(-\tau) \\
		-\lambda \phi_{2}(0) \\
		-\lambda \phi_{3}(0)
	\end{bmatrix}
    \text{ and }
    B y := \begin{bmatrix}
    	y\\
    	0\\
    	0
    \end{bmatrix}.
\end{equation}
The transfer function (which was also used in \cite{Smith1992}) corresponding to the linear part of \eqref{EQ: GoodwinRhoBeta} is given by 
\begin{equation}
W_{\rho}(p):=-\frac{1}{(\lambda + p)^{3} e^{p \tau} + \rho}.
\end{equation}

We have the following theorem.
\begin{theorem}
	\label{TH: DelayFirstTheorem}
	Suppose that for some $\beta > 1$ there exists $\rho \in (-\delta_{\beta},\kappa_{0}]$ such that
	\begin{enumerate}
		\item[\textbf{(DF1)}] $\rho \tau^{3} e^{\lambda \tau} < 84.2$.
		\item[\textbf{(DF2)}] $\operatorname{Re} \left[(1 + (-\kappa_{0} + \rho) W_{\rho}(i\omega - \lambda) )^{*} ( 1 + (\rho + \delta_{\beta}) W_{\rho}(i\omega - \lambda) ) \right] > 0$ for all $\omega \in \mathbb{R}$.
	\end{enumerate}
    Then there exists an operator $P$ such that \textbf{(H1)}, \textbf{(H2)} with $j=2$ and \textbf{(H3)} are satisfied for the semi-flow $\varphi_{\beta}$ generated by \eqref{EQ: DelayGoodwinBeta}.
\end{theorem}
Note that in Theorem \ref{TH: DelayFirstTheorem} we used a more delicate frequency-domain condition \textbf{(DF2)}, which uses the both bounds $-\kappa_{0}+\rho \leq F'_{\rho,\beta}(\sigma) \leq \rho+\delta_{\beta}$, where $\sigma \in \mathbb{R}$, for the derivative of $F_{\rho,\beta}$. The proof is the same as for Theorem \ref{TH: DelayODEFreqCond} and follows from results in \cite{Anikushin2020FreqDelay}. The spectral properties of the linear part are established in the proof of Theorem 15 from \cite{Smith1992}.

\begin{remark}
	\label{REM: SmithDelayRemark}
	In \cite{Smith1992} instead of conditions \textbf{(DF1)} and \textbf{(DF2)} it was used the more concrete condition
	\begin{equation}
		\label{EQ: SmithFreqIneqGoodwin}
		\kappa_{0} \tau^{3} e^{\lambda \tau} < 84.2
	\end{equation}
    and $\rho= \frac{1}{2} \kappa_{0} + \delta'$, where $\delta'>0$ is sufficiently small such that $\rho$ satisfies $\frac{1}{2} \kappa_{0} < \rho < \frac{1}{2} 84.2 \cdot \tau^{-3}e^{-\lambda \tau}$ (and \textbf{(DF1)} in particular). It is also shown that for these $\rho$'s we have the inequality (here $|\cdot|$ is the usual norm in $\mathbb{C}$)
    \begin{equation}
    |W_{\rho}(-\lambda + i\omega)| < \rho^{-1} \text{ for all } \omega \in \mathbb{R}.
    \end{equation}
    and, consequently, \eqref{EQ: FreqDomainDelaySmith} is satisfied since $\rho$ is a Lipschitz constant of $F_{\rho,\beta}$. It can be shown that in this case \textbf{(DF2)} is also satisfied, i.~e. it is more flexible\footnote{This is not surprising since \textbf{(DF2)} takes into account the lower and upper bounds of the derivative of $F_{\beta,\rho}$ and \eqref{EQ: FreqDomainDelaySmith} uses only the maximum of modules.}. However, in the concrete example below (see Fig. \ref{FIG: DFreq}) having a varying parameter $\rho$ do not give any improvements (at least visually) comparing with the results of \cite{Smith1992}. Moreover, in the statement of our Theorem \ref{TH: DelayFirstTheorem} the parameter $\beta$ is fixed. Under \eqref{EQ: SmithFreqIneqGoodwin} there is no dependence on $\beta$.
\end{remark}

It is clear that the only stationary point of \eqref{EQ: GoodwinRhoBeta}, which lies in the cone of non-negative functions, is $\phi_{0} \equiv (\lambda^{2} \eta_{0}, \lambda \eta_{0}, \eta_{0})=g(\eta_{0})(\lambda^{-1},\lambda^{-2},\lambda^{-3})$, where $\eta_{0}$ satisfy $g(\eta_{0}) = \lambda^{3}\eta_{0}$. Due to the monotonicity of $g$ such $\eta_{0}$ is unique. Note that in virtue of the inequality $g(0) > g(\eta_{0}) > g(\sigma_{\beta})$ the point $\phi_{0}$ lies in $\mathcal{W}_{\beta}$ for all $\beta \geq 1$. Let $0<\theta_{1}<\pi/2$ be the unique number such that $\tau \lambda \tan(\theta_{1}) = \pi - 3 \theta_{1}$. Let $\overline{\mathcal{W}}$ denote the closure of $\mathcal{W}_{1}$ in $\mathbb{E}$.

Analysis of the spectrum for the linearization of \eqref{EQ: GoodwinRhoBeta} at $\phi_{0}$ (see Theorem 15 in \cite{Smith1992} for all calculations) gives the following result.
\begin{theorem}
	\label{TH: DelayEqsFinalTheorem}
	Suppose the assumptions of Theorem \ref{TH: DelayFirstTheorem} hold. If
	\begin{equation}
	\label{EQ: TerminalDelayCondition}
	g'(\eta_{0}) \not= -(\lambda \sec \theta_{1})^{3}
	\end{equation}
	is satisfied then for every $\phi_{0} \in \mathcal{W}_{\beta}$ the $\omega$-limit set $\omega(\phi_{0})$ is either the stationary point $\phi_{0}$ or a periodic trajectory in $\overline{\mathcal{W}}$. If, in addition, $g'(\eta_{0}) < -(\lambda \sec\theta_{1})^{3}$ then there exists at least one periodic orbit in $\overline{\mathcal{W}}$, which is orbitally stable.
\end{theorem}
Condition \eqref{EQ: TerminalDelayCondition} shows that $\phi_{0}$ is hyperbolic and since its unstable manifold can be only of even dimension (see \cite{Smith1992}), it is terminal. Thus, the first part of the theorem follows from Theorem \ref{TH: PBTrichotomy} and Corollary \ref{COR: TerminalPoints}. The condition $g'(\eta_{0}) < -(\lambda \sec\theta_{1})^{3}$ implies that the unstable manifold at $\phi_{0}$ is two-dimensional and, consequently, Theorem \ref{TH: PeriodicOrbitExis} is applicable.

Now as in \cite{Smith1992} we consider a special case of \eqref{EQ: DelayGoodwin} with $g(\sigma)= ( 1 + |\sigma|^{3} )^{-1}$. Fig. \ref{FIG: DFreq} shows a numerically obtained region in the space of parameters $(\tau,\lambda)$ for which the conditions of Theorem \ref{TH: DelayEqsFinalTheorem} are satisfied. In this case having the varying parameter $\rho$ seems to cause no improvements (comparing with \cite{Smith1992}) since these regions visually coincide with the one analytically given by \eqref{EQ: SmithFreqIneqGoodwin}.
\begin{figure}
	\centering
	\includegraphics[width=1.\linewidth]{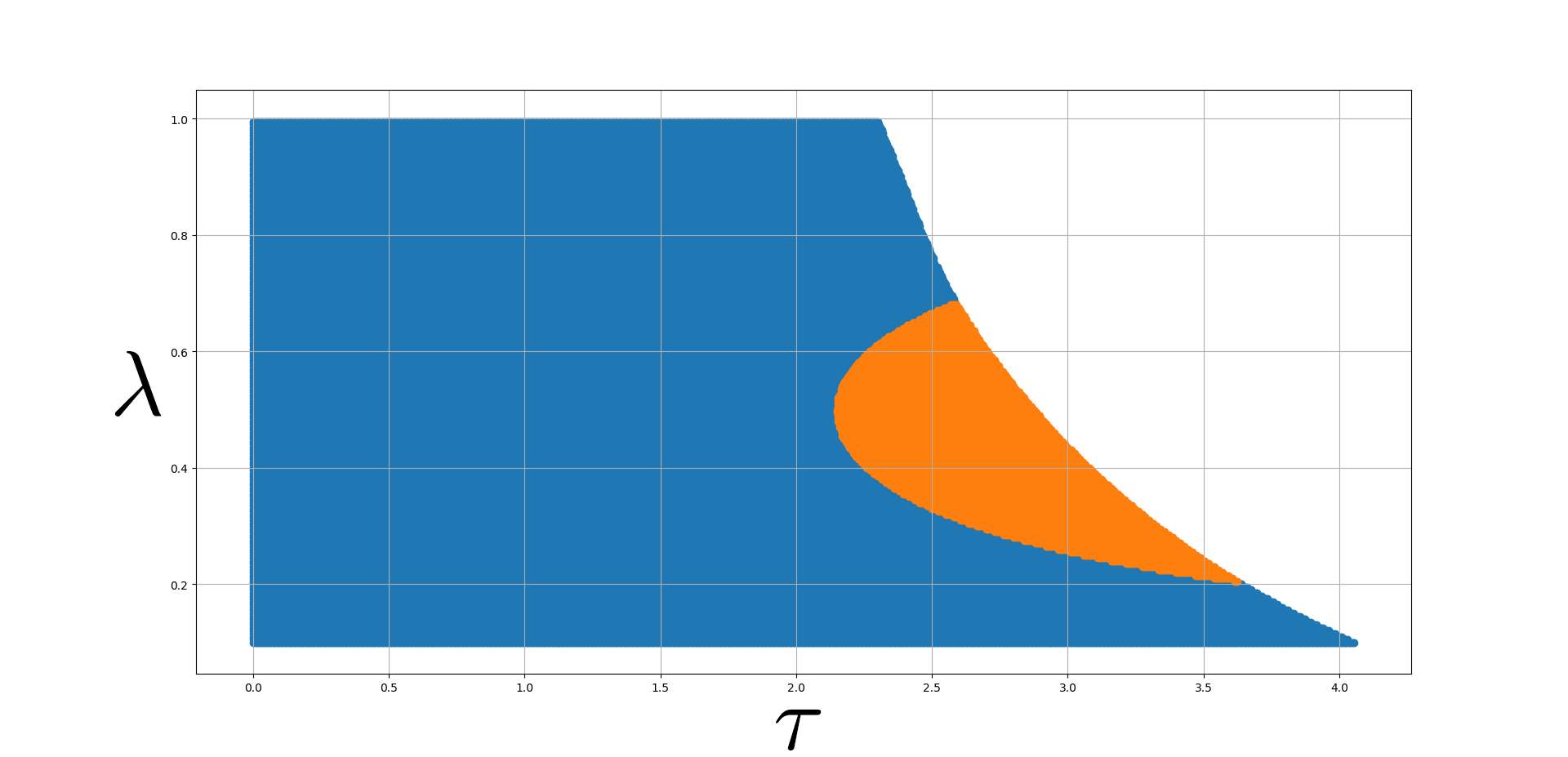}	
	\caption{A numerically obtained region in the space of parameters $(\tau,\lambda)$ of system \eqref{EQ: DelayGoodwin} with $g(\sigma)= ( 1 + |\sigma|^{3} )^{-1}$, for which the conditions of Theorem \ref{TH: DelayEqsFinalTheorem} are satisfied. The blue region corresponds to the case of a stable stationary point. In the orange region the stationary point has a local two-dimensional unstable manifold and therefore the existence of an orbitally stable periodic orbit is guaranteed. Note that the orange region visually coincides with Fig. 2 from \cite{Smith1992}.}
	\label{FIG: DFreq}
\end{figure}

%% file: ApplicationsParab.tex
\section{Abstract semilinear parabolic equations}
\label{SEC: ReactionDiffusionSmith}
Let $A \colon \mathcal{D}(A) \subset \mathbb{H}_{0} \to \mathbb{H}_{0}$ be a self-adjoint positive-definite operator acting in a real Hilbert space $\mathbb{H}_{0}$ and such that $A^{-1} \colon \mathbb{H}_{0} \to \mathbb{H}_{0}$ is compact. Then one can define powers of $A$ and a scale of Hilbert spaces $\mathbb{H}_{\alpha}:=\mathcal{D}(A^{\alpha})$ with the scalar products $( v_{1},v_{2})_{\alpha}:=(A^{\alpha}v_{1},A^{\alpha}v_{2})$ for all $v_{1},v_{2} \in \mathcal{D}(A^{\alpha})$ (see, for example, \cite{Chueshov2015}). Let $\alpha \in [0,1)$ be fixed and let $\mathbb{M},\Xi$ be some Hilbert spaces. Suppose that linear bounded operators $C \colon \mathbb{H}_{\alpha} \to \mathbb{M}$ and $B \colon \Xi \to \mathbb{H}_{0}$ are given. Let $F \colon \mathbb{R} \times \mathbb{M} \to \Xi$ be a continuous map such that for some constant $\Lambda>0$ and all $\sigma_{1},\sigma_{2} \in \mathbb{M}$, $t \in \mathbb{R}$ we have
\begin{equation}
	|F(\sigma_{1})-F(\sigma_{2})|_{\Xi} \leq \Lambda |\sigma_{1}-\sigma_{2}|_{\mathbb{M}}.
\end{equation}

Finally, let $K \in \mathcal{L}(\mathbb{H}_{0}) \cap \mathcal{L}(\mathbb{H}_{\alpha})$ be a bounded linear operator\footnote{We discuss the case of bounded $K$ for simplicity. For some problems, consideration of unbounded operators $K$, which are dominated by $A$, may be useful.}. We consider the nonlinear evolution equation
\begin{equation}
	\label{EQ: ParabolicNonAutEquation}
	\dot{v}(t) = (-A+K)v(t) + BF(Cv(t)).
\end{equation}
Following the fixed point arguments of Theorem 4.2.3 in \cite{Chueshov2015}, it can be shown that for any $v_{0} \in \mathbb{H}_{\alpha}$ there exists a unique mild solution $v(t)=v(t,v_{0})$ defined for $t \geq 0$ and such that $v(0)=v_{0}$. Moreover, for every $T>0$ there is a constant $C=C(T)>0$ such that the estimate 
\begin{equation}
	|v(t,v_{1}) - v(t,v_{2})|_{\alpha} \leq C |v_{1}-v_{2}|_{\alpha}
\end{equation}
holds for all $v_{1},v_{2} \in \mathbb{H}_{\alpha}$, and $t \in [0,T]$. From this it follows that $\varphi^{t}(v_{0}):=v(t,v_{0})$, where $t \geq 0$ and $v_{0} \in \mathbb{H}_{\alpha}$, is a semiflow in $\mathbb{E}=\mathbb{H}:=\mathbb{H}_{\alpha}$. Moreover, the map $\varphi^{t} \colon \mathbb{H}_{\alpha} \to \mathbb{H}_{\alpha}$ is compact for all $t > 0$, i.~e. \textbf{(COM)} is satisfied with any $\tau_{com} > 0$.

Since $-A$ generates a $C_{0}$-semigroup in $\mathbb{H}$ and $\mathbb{H}_{\alpha}$, the operator $-A+K$ also is a generator of a $C_{0}$-semigroup (see \cite{EngelNagel2000}). Below we consider numbers $\nu > 0$ such that the operator $-A+K + \nu I$ admits an exponential dichotomy in $\mathbb{H}_{\alpha}$  with the stable space $\mathbb{H}^{s}_{\alpha}(\nu)$ and the unstable space $\mathbb{H}^{u}_{\alpha}(\nu)$.

Let us consider the transfer operator $W(p):=C(-A+K-pI)^{-1}B$ (to be more precise, one should consider complexifications of the corresponding operators).

The following theorem follows from results in Section 5 from \cite{Anikushin2020FreqParab}.
\begin{theorem}
	\label{TH: ParabSemilinFreqCond}
	Suppose there is $\nu > 0$ such that the operator $-A+K + \nu I$ admits an exponential dichotomy in $\mathbb{H}_{\alpha}$ such that $\mathbb{H}_{\alpha}= \mathbb{H}^{s}_{\alpha}(\nu) \oplus \mathbb{H}^{u}_{\alpha}(\nu)$ and $\dim \mathbb{H}^{u}_{\alpha} = j$. Moreover, let the frequency-domain condition
	\begin{equation}
	\label{EQ: FreqParabGeneral}
    \|W(-\nu + i\omega)\|_{\Xi^{\mathbb{C}} \to \mathbb{M}^{\mathbb{C}}} < \Lambda^{-1}, \text{ for all } \omega \in \mathbb{R}.
	\end{equation}
	be satisfied. Then there exists an operator $P \in \mathcal{L}(\mathbb{H}_{\alpha})$ such that the semiflow $\varphi$ in $\mathbb{E} = \mathbb{H}_{\alpha}$ generated by \eqref{EQ: ParabolicNonAutEquation} satisfies \textbf{(H1)} with $\mathbb{E}^{+}= \mathbb{H}^{s}_{\alpha}(\nu)$ and $\mathbb{E}^{-}=\mathbb{H}^{u}_{\alpha}(\nu)$; \textbf{(H2)} with the given $j$ and \textbf{(H3)} with the given $\nu$, $\mathbb{H}=\mathbb{H}_{\alpha}$, some $\delta > 0$ and $\tau_{V} =0$.
\end{theorem}
In \cite{Smith1994PB,Smith1994} R.~A.~Smith studied certain reaction-diffusion systems (i.~e. when $\alpha = 0$), including the FitzHugh-Nagumo system and diffusive Goodwin equations, in domains $\Omega \subset \mathbb{R}^{d}$ with $d \leq 3$. His frequency-domain condition is a particular case of \eqref{EQ: FreqParabGeneral}. Moreover, under his assumptions the measurement space $\mathbb{M}$ cannot be finite-dimensional, which excludes, for example, the case when $C$ is a linear functional.

Now let us consider the simplest case when $K=0$, $\mathbb{M} = \mathbb{H}_{\alpha}$, $\Xi = \mathbb{H}_{0}$ and $B=\operatorname{Id}$, $C=\operatorname{Id}$. Let $0 < \lambda_{1} \leq \lambda_{2} \leq \ldots$ be the eigenvalues of $A$. Let us take $j$ such that $\lambda_{j} < \lambda_{j+1}$ and $\nu \in (\lambda_{j},\lambda_{j+1})$. Then \eqref{EQ: FreqParabGeneral} takes the form
\begin{equation}
	\label{EQ: SpetrGapConResolventIneq}
	\| (A+(-\nu + i \omega)I)^{-1} \|_{\mathbb{H}^{\mathbb{C}}_{0} \to \mathbb{H}^{\mathbb{C}}_{\alpha}} \leq \Lambda^{-1} \text{ for all } \omega \in \mathbb{R}.
\end{equation}
One can show (see \cite{Anikushin2020FreqParab}) that there is $\nu \in (\lambda_{j},\lambda_{j+1})$, which minimizes the left-hand side of \eqref{EQ: SpetrGapConResolventIneq}. The corresponding inequality takes the form
\begin{equation}
	\frac{\lambda_{j+1}-\lambda_{j}}{\lambda^{\alpha}_{j}+\lambda^{\alpha}_{j+1}} > \Lambda,
\end{equation}
which is known as the Spectral Gap Condition \cite{Temam1997}, which guarantees the existence of $j$-dimensional inertial manifolds for \eqref{EQ: ParabolicNonAutEquation}. From this it is clear that the general frequency-domain condition given by \eqref{EQ: FreqParabGeneral} provides a lot of flexibility. The control operator $B$ cuts the input of the resolvent and the measurement operator $C$ cuts the output that sometimes makes the norm in \eqref{EQ: FreqParabGeneral} much smaller than the norm of a single resolvent. The presence of the operator $K$ may help to achieve required spectral properties or to provide more sharp frequency-domain conditions.

In the case of parabolic equations studying of stationary points, which may be inhomogeneous in space, is a challenging task. Thus, applications of Theorem \ref{TH: PeriodicOrbitExis} are restricted to the cases, in which we know all the stationary points (at least in a some bounded region). In \cite{Smith1994PB,Smith1994} R.~A.~Smith used a simple criterion, which guarantees that the stationary point is unique and, consequently, it may be easily calculated if it is space-homogeneous.

%% file: ApplicationsOther.tex
\section{Other applications and further developments}
\label{SEC: OtherApps}
The presented theory is potentially applicable to parabolic equations with delay. It is known (see, for example, the monograph of I.~Chueshov \cite{Chueshov2015}) that such equations are well-posed in the space $\mathbb{E} = C([-\tau,0];\mathbb{H}_{\alpha})$, where $\mathbb{H}_{\alpha}$ is defined in the previous section, and the corresponding semiflow satisfies \textbf{(COM)}. The frequency theorem from \cite{Anikushin2020FreqDelay} allows to construct quadratic functionals for certain reaction-diffusion systems ($\alpha = 0$) with delay. Some clarifications in \cite{Anikushin2020FreqDelay} are required to cover the general case $\alpha \in [0,1)$ (such a modification should cover both versions of the frequency theorem from \cite{Anikushin2020FreqParab,Anikushin2020FreqDelay}). The well-posedness of such equations in the Hilbert space $\mathbb{H} = \mathbb{H}_{\alpha} \times L_{2}(-\tau,0;\mathbb{H}_{\alpha})$ also requires further developments. A simple approach is suggested in \cite{Anikushin2020Semigroups}, where one may use the well-posedness in  $C([-\tau,0];\mathbb{H}_{\alpha})$, embed the semiflow in $\mathbb{H}$ and then show that it can be uniquely extended by continuity to a semiflow in $\mathbb{H}$. For applications of Theorem \ref{TH: PeriodicOrbitExis} local analysis of stationary states is required. The present author does not know sources, where the Stable/Unstable Manifold Theorem for such equations is proved. In Y.~Tan et. al \cite{TanHuangWang2018} it is considered some epidemiological model given by a certain class of reaction-diffusion systems with delay and stability analysis is presented. It may be interesting to obtain for such a system conditions for the existence of an orbitally stable periodic orbit using this theory.

Another field of applications is parabolic equations with boundary controls (nonlinear boundary conditions). Here we can use the frequency theorem from \cite{Likhtarnikov1976} to construct quadratic functionals for such problems (see \cite{Anikushin2020Red} for a simple application). However, \cite{Likhtarnikov1976} contains restrictive stabilizability and regularity assumptions, which should be relaxed in the spirit of recent versions of the frequency theorem. This class, to the best of our knowledge, also requires developments of methods for studying of local analysis and well-posedness. For parabolic boundary control problems with delay the situation is much harder.

%% file: Acknowledgements.tex
\section*{Acknowledgements}

I thank V.~Reitmann for many useful discussions on the topic and A.~O.~Romanov for helping in writing the Python program with which Fig. \ref{FIG: DFreq} was obtained.